\documentclass{article}
\usepackage{amssymb}
\usepackage{mathtools}
\mathtoolsset{showonlyrefs}

\usepackage{amsthm} 

\usepackage{ifxetex,ifluatex}
\ifxetex
\usepackage{xltxtra}
\else
\ifluatex
\usepackage{luatextra}
\else
\usepackage[utf8]{inputenc}
\usepackage[T1]{fontenc}
\usepackage{newtxtext}
\usepackage[vvarbb]{newtxmath}
\csname else\expandafter\endcsname
\romannumeral -`0
\fi
\fi
\iftrue\relax%
\defaultfontfeatures{Ligatures=TeX}
\usepackage[math-style=TeX]{unicode-math}
\AtBeginDocument{%
}
\fi
\providecommand{\allOne}{\mathbb{1}}

\extrafloats{100}

\usepackage{booktabs}       
\usepackage{microtype}      
\usepackage{natbib}         

\usepackage{authblk}

\usepackage{etoolbox}
\usepackage{easy-todo}
\usepackage{xspace}
\usepackage{algorithm}
\usepackage{algorithmic}
\usepackage{enumitem}

\usepackage{geometry}
\usepackage[hyphens]{url}   
\usepackage{hyperref} 
\usepackage{color}
\usepackage{graphicx}

\hypersetup{
  pdftitle={Blending Conditional Gradients algorithm},
  pdfauthor={Gábor Braun, Sebastian Pokutta, Dan Tu, Stephen Wright},
  pdfkeywords={Frank–Wolfe algorithm, conditional gradient,
    gradient descent,
    caching, linear optimization oracle, convex optimization},
  pdfsubject={MSC2000 Primary 68Q32; 
    Secondary
    90C52 
  }}

\title{Blended Conditional Gradients: \\ the unconditioning
  of conditional gradients}

\author[1]{Gábor Braun}
\author[1]{Sebastian Pokutta}
\author[1]{Dan Tu}
\author[2]{Stephen Wright}
\affil[1]{ISyE, Georgia Institute of Technology, Atlanta, GA\\
  \texttt{\{gabor.braun,sebastian.pokutta\}@isye.gatech.edu,
    dan.tu@gatech.edu}}
\affil[2]{Computer Sciences Department, University of Wisconsin\\
  Madison, WI\\
  \texttt{swright@cs.wisc.edu}}

\newcommand{\sidebyside}[2]{%
  \par
  \hfil
  \includegraphics{graphics/#1}
  \hfil
  \includegraphics{graphics/#2}
  \hfil
  \par}

\setlist[enumerate]{label=(\roman*)}
\newlist{enumerate*}{enumerate*}{1}
\setlist[enumerate*]{label=(\arabic*),
  after=.,
  itemjoin={{, }}, itemjoin*={{, or }}}

\floatstyle{ruled}
\newfloat{oracle}{htbp}{ora}
\floatname{oracle}{Oracle}

\newcommand{\LPsep}[1]{\ensuremath{\operatorname{LPsep}\sb{#1}}}

\newcommand{\Sora}{\TextOrMath{\hyperref[ora:simplex]{Oracle
      \ensuremath{\operatorname{SiDO}}}}{\operatorname{SiDO}}\xspace}

\newtheorem{theorem}{Theorem}[section]
\newtheorem{fact}[theorem]{Fact}

\newtheorem{corollary}[theorem]{Corollary}
\newtheorem{lemma}[theorem]{Lemma}
\theoremstyle{definition}

\theoremstyle{remark}

\DeclarePairedDelimiterX{\normSimple}[1]{\lVert}{\rVert}
{\ifblank{#1}{\mathord{\cdot}}{#1}}

\DeclarePairedDelimiter{\size}{\lvert}{\rvert}

\newcommand{\diam}[1]{\operatorname{diam}(#1)} 

\newcommand{\norm}[2][]{\normSimple{#2}\ifblank{#1}{}{\sb{#1}}}

\newcommand{\dualnorm}[2][]{\mathinner{%
    \normSimple{#2}\sb{\ifblank{#1}{}{#1,}*}}}

\DeclarePairedDelimiter{\card}{\lvert}{\rvert}
\DeclarePairedDelimiterXPP{\supp}[1]{\operatorname{supp}}(){}{#1}

\newcommand{\R}{\mathbb R}

\newcommand{\NPhi}{N_{\Phi}}
\newcommand{\Nd}{N_{\text{desc}}}
\newcommand{\Td}{T_{\text{desc}}}

\newcommand{\trace}[1]{\operatorname{Tr}\left[{#1}\right]}

\DeclareMathOperator*{\conv}{conv}
\DeclareMathOperator*{\argmax}{argmax}
\DeclareMathOperator*{\argmin}{argmin}

\begin{document}
\maketitle{}
%

\begin{abstract}
  We present a \emph{blended conditional gradient} approach for
  minimizing a smooth convex function over a polytope $P$, combining
  the Frank–Wolfe algorithm (also called conditional gradient) with
  gradient-based steps, different from away steps and pairwise steps,
  but still achieving linear convergence for strongly convex
  functions, along with good practical performance. Our approach
  retains all favorable properties of conditional gradient algorithms,
  notably avoidance of projections onto $P$ and maintenance of
  iterates as sparse convex combinations of a limited number of
  extreme points of $P$.  The algorithm is \emph{lazy}, making use of
  inexpensive inexact solutions of the linear programming subproblem that
  characterizes the conditional gradient approach.  It decreases
  measures of optimality rapidly, both in the
  number of iterations and in wall-clock time, outperforming even the
  lazy conditional gradient algorithms of \cite{braun2016lazifying}.
  We also present a streamlined version of the algorithm for
  the probability simplex.
\end{abstract}

\section{Introduction}
\label{sec:introduction}

A common paradigm in convex optimization is
minimization of a smooth convex function \(f\) over a polytope $P$.
The conditional gradient (CG) algorithm, also known as ``Frank–Wolfe''
\cite{frank1956algorithm}, \cite{levitin1966constrained} is enjoying
renewed popularity because it can be implemented efficiently to
solve important problems in data analysis. It is a first-order method,
requiring access to gradients $\nabla f(x)$ and function values
$f(x)$.  In its original form, CG employs a linear programming (LP)
oracle to minimize a linear function over the polytope $P$ at each
iteration.  The cost of this operation depends on the complexity of
\(P\). The base method has many extensions with the aim of improving
performance, like reusing previously found points of \(P\)
to complement or even sometimes omit LP oracle calls
\cite{FW-converge2015}, or
using oracles weaker than an LP oracle to reduce cost of oracle calls
\cite{braun2016lazifying}.

In this work, we describe a \emph{blended conditional gradient (BCG)}
approach, which is a novel combination of
several previously used ideas into a single algorithm
with similar theoretical convergence rates
as several other variants of CG that have been
studied recently, including pairwise-step and away-step variants and
the lazy variants in \cite{braun2016lazifying}, however, with very fast performance and in several
cases, empirically higher convergence rates compared to
other variants. In particular, while the lazy variants of \cite{braun2016lazifying}
have an advantage over baseline CG when the LP oracle is expensive, our
BCG approach consistently outperforms the other variants in more
general circumstances, both in per-iteration progress and in
wall-clock time.

In a nutshell, BCG is a first-order algorithm that chooses among
several types of steps based on the gradient $\nabla f$ at the
current point.  It also maintains an ``active vertex set'' of
solutions from previous iterations, like e.g., the Away-step
Frank–Wolfe algorithm.  Building on \cite{braun2016lazifying}, BCG
uses a ``weak-separation oracle'' to find a vertex of $P$ for which
the linear objective attains some fraction of the reduction in $f$
the LP oracle would achieve, typically by first
searching among the current set of active vertices and
if no suitable vertex is found,
the LP oracle used in the original FW algorithm may be
deployed.  On other iterations, BCG employs a ``simplex descent
oracle,'' which takes a step within the convex hull of the active
vertices, yielding progress either via reduction in function value (a
``descent step'') or via culling of the active vertex set (a ``drop
step'').  For example, the oracle can make a single (vanilla) gradient
descent step.  The size of the active vertex set typically remains
small, which benefits both the efficiency of the method and the
``sparsity'' of the final solution (i.e., its representation as a convex
combination of a relatively small number of vertices). Compared to the
Away-step and Pairwise-step Frank–Wolfe algorithms, the simplex
descent oracle realizes an improved use of the active set by
(partially) optimizing over its convex hull via descent steps, similar
to the Fully-corrective Frank–Wolfe algorithm and also the approach in
\cite{rao2015forward} but with a better step selection criterion: BCG
alternates between the various steps using estimated progress from
dual gaps. We hasten to stress that BCG remains projection-free.

\subsection*{Related work}
\label{sec:related-work}

There has been an extensive body of work on conditional gradient algorithms;
see the excellent overview of \cite{jaggi2013revisiting}. Here we
review only those papers most closely related to our work.

Our main inspiration comes from
\cite{braun2016lazifying,lan2017conditional}, which introduce the
weak-separation oracle as a lazy alternative to calling the
LP oracle in every iteration. It is influenced too by the method of
\cite{rao2015forward}, which maintains an active vertex set, using
projected descent steps to improve the objective over the convex hull
of this set, and culling the set on some steps to keep its size under
control.
While the latter method is a heuristic
with no proven convergence bounds
beyond those inherited from the standard Frank–Wolfe method,
our BCG algorithm employs a criterion for \emph{optimal trade-off between the various
steps}, with a \emph{proven} convergence rate equal to
state-of-the-art Frank–Wolfe variants up to a constant factor.

Our main result shows linear convergence of BCG for strongly convex
functions.  Linearly convergent variants of CG were studied as early
as \cite{guelat1986some} for special cases and
\cite{garber2013linearly} for the general case (though the latter work
involves very large constants).  More recently, linear convergence has
been established for various pairwise-step and away-step variants of
CG in \cite{FW-converge2015}. Other memory-efficient
decomposition-invariant variants were described in \cite{LDLCC2016}
and \cite{bashiri2017decomposition}.  Modification of descent
directions and step sizes, reminiscent of the drop steps used in BCG,
have been considered by \cite{Freund2016,freund2017extended}.  The use
of an inexpensive oracle based on a subset of the vertices of $P$, as
an alternative to the full LP oracle, has been considered in
\cite{kerdreux2018frank}. \cite{garber2018fast} proposes a fast
variant of conditional gradients for matrix recovery problems.

BCG is quite distinct from the Fully-corrective Frank–Wolfe algorithm
(FCFW) (see, for example,
\cite{holloway1974extension,FW-converge2015}).
Both approaches maintain active vertex sets, generate iterates that
lie in the convex hulls of these sets, and alternate between
Frank–Wolfe steps generating new vertices and correction
steps optimizing within the current active vertex set. However,
convergence analyses of the FCFW algorithm assume that the correction
steps have unit cost, though they can be quite expensive in practice.
For BCG, by
contrast, we assume only a \emph{single} step of gradient descent type
having unit cost (disregarding cost of line search).
For computational results comparing BCG and FCFW, and illustrating
this issue, see
Figure~\ref{fig:fcfwComp} and the discussion in
Section~\ref{sec:comp-results}.

\subsection*{Contribution}
\label{sec:contribution}

Our contribution is summarized as follows:

\begin{description}
\item[Blended Conditional Gradients (BCG).]
  The BCG approach blends different types of descent steps:
  Frank–Wolfe steps from \cite{frank1956algorithm},
  optionally lazified
  as in \cite{braun2016lazifying}, and gradient descent steps over
  the convex hull of the current active vertex set. It avoids
  projections
  and does not use away steps and pairwise steps, which are elements
  of other popular variants of CG.  It achieves linear convergence for
  strongly convex functions (see Theorem~\ref{thm:pgd-FW}), and
  \(O(1/t)\) convergence after \(t\) iterations for general smooth
  functions.  While the linear convergence proof of the Away-step
  Frank–Wolfe Algorithm \cite[Theorem~1, Footnote~4]{FW-converge2015}
  requires the objective function \(f\) to be defined on the Minkowski
  sum \(P - P + P\), BCG does not need \(f\) to be defined outside the
  polytope \(P\).  The algorithm has complexity comparable to
  pairwise-step or away-step variants of conditional gradients,
  both in time measured as number of iterations and in space
  (size of active set).
  It is affine-invariant and parameter-free; estimates of
  such parameters as smoothness, strong convexity, or the diameter of
  \(P\) are not required.  It maintains iterates as (often sparse)
  convex combinations of vertices, typically much sparser than the
  baseline CG methods, a property that is important for some
  applications. Such sparsity is due to the aggressive reuse of active
  vertices, and the fact that new vertices are added only as a kind of
  last resort.  In wall-clock time as well as per-iteration progress,
  our computational results show that BCG can be orders of magnitude
  faster than competimg CG methods on some problems.

\item[Simplex Gradient Descent (SiGD).]
  In Section~\ref{sec:proj-grad-desc-projfree}, we describe a new
  projection-free gradient descent procedure for minimizing a smooth
  function over the probability simplex, which can be used to
  implement the ``simplex descent oracle'' required by BCG,
  which is the module doing gradient descent steps.

\item[Computational Experiments.]
  We demonstrate the excellent computational behavior of BCG compared
  to other CG algorithms on standard problems, including video
  co-localization, sparse regression, structured SVM training, and
  structured regression.  We observe significant computational
  speedups and in several cases empirically better convergence rates.
\end{description}

\subsection*{Outline}
\label{sec:outline}

We summarize preliminary material in Section~\ref{sec:preliminaries},
including the two oracles that are the foundation of our BCG
procedure.  BCG is described and analyzed in
Section~\ref{sec:offline-pgd-lazy}, establishing linear convergence
rates.  The simplex gradient descent routine, which implements the
simplex descent oracle, is described in
Section~\ref{sec:proj-grad-desc-projfree}. We mention in
particular a variant of BCG that applies when $P$ is the probability
simplex, a special case that admits several simplifications and
improvements to the analysis.
Some possible enhancements to BCG are discussed in
Section~\ref{sec:impr-actu-impl}.
Our computational
experiments appear in
Section~\ref{sec:comp-results}.

\section{Preliminaries}
\label{sec:preliminaries}

We use the following notation: \(e_{i}\) is the \(i\)-th coordinate
vector, \(\allOne \coloneqq (1, \dots, 1) = e_{1} + e_{2} + \dotsb\)
is the all-ones vector, \(\norm{\cdot}\) denotes the Euclidean norm
(\(\ell_{2}\)-norm), \(D = \diam{P} = \sup_{u,v \in P} \norm[2]{u-v}\)
is the \(\ell_2\)-diameter of
\(P\), and \(\conv S\) denotes the convex hull of a set \(S\) of
points.  The \emph{probability simplex} \(\Delta^{k} \coloneqq \conv\{
e_{1}, \dotsc, e_{k}\}\) is the convex hull of the coordinate vectors
in dimension $k$.

Let \(f\) be a differentiable convex function.
Recall that $f$ is
\emph{$L$-smooth} if
\begin{equation}
  \label{eq:L}
  f(y) - f(x) - \nabla f(x) (y-x) \leq L \norm{y-x}^{2} / 2
  \quad
  \text{for all  \(x, y \in P\).}
\end{equation}
The function \(f\) has \emph{curvature \(C\)} if
\begin{equation}
  \label{eq:C}
  f(\gamma y + (1-\gamma) x) \leq f(x) +
  \gamma \nabla f(x) (y - x) + C \gamma^{2} / 2
  \quad
  \text{for all \(x,y \in P\) and \(0 \leq \gamma \leq 1\).}
\end{equation}
(Note that an \(L\)-smooth function always has curvature \(C \leq LD^2\).)
Finally, $f$ is \emph{strongly convex} if for some \(\alpha > 0\) we have
\begin{equation}
  \label{eq:sc}
  f(y) - f(x) - \nabla f(x) (y-x) \geq \alpha \norm{y-x}^{2} / 2
  \quad
  \text{for all \(x, y \in P\).}
\end{equation}
 We will use the following fact about strongly convex function when
 optimizing over \(P\).
\begin{fact}[Geometric strong convexity guarantee]\citep[Theorem~6 and
  Eq.~(28)]{FW-converge2015}
  \label{fact:strongConvP}
Given a strongly convex function \(f\), there is a value \(\mu > 0\)
called the \emph{geometric strong convexity} such that
\begin{equation}
  \label{eq:geometric-strong-convex}
  f(x) - \min_{y \in P} f(y)
  \leq \frac{\left(
      \max_{y \in S, z \in P} \nabla f(x) (y - z)
    \right)^{2}}{2 \mu}
\end{equation}
for any \(x \in P\) and for any subset \(S\) of the vertices of \(P\)
for which \(x\) lies in the convex hull of \(S\).
\end{fact}
The value of \(\mu\) depends both on \(f\) and the geometry of
\(P\). For example, a possible choice is decomposing $\mu$ as a
product of the form \(\mu = \alpha_f W_P^2\), where \(\alpha_f\) is
the strong convexity constant of \(f\) and \(W_P\) is the
\emph{pyramidal width of \(P\)}, a constant only depending on the
polytope \(P\);
see \cite{FW-converge2015}.

\subsection{Simplex Descent Oracle}
\label{sec:proj-grad-desc}

Given a convex objective function \(f\) and
an ordered finite set \(S = \{v_{1}, \dotsc, v_{k}\}\) of points,
we define \(f_{S} \colon \Delta^{k} \to
\mathbb{R}\) as follows:
\begin{equation} \label{eq:def.fS}
  f_{S}(\lambda) \coloneqq f\left( \sum_{i=1}^{k} \lambda_{i} v_{i}
  \right).
\end{equation}
When \(f_{S}\) is  \(L_{f_{S}}\)-smooth,
Oracle~\ref{ora:simplex} returns an improving point \(x'\)
in \(\conv S\) together with a vertex set \(S' \subseteq S\) such that
\(x' \in \conv S'\).

\begin{oracle}
  \caption{Simplex Descent Oracle \(\Sora(x, S, f)\)}
  \label{ora:simplex}
  \begin{algorithmic}
    \REQUIRE finite set \(S \subseteq \mathbb{R}^{n}\), point \(x \in
      \conv S\),
      convex smooth function \(f \colon \conv S \to \mathbb{R}^{n}\);
    \ENSURE finite set \(S'
    \subseteq S\), point \(x' \in \conv S'\) satisfying either
      \begin{description}
      \item[drop step:]
        \(f(x') \leq f(x)\)
        and \(S' \neq S\)
      \item[descent step:]
        \mbox{}\\[.5em] 
        \(f(x) - f(x') \geq [\max_{u,v \in S} \nabla f(x) (u - v)]^{2}
        / (4 L_{f_{S}})\)
      \end{description}
\end{algorithmic}
\end{oracle}

In Section~\ref{sec:proj-grad-desc-projfree} we provide an
implementation (Algorithm~\ref{alg:simplex-descent}) of this oracle
via a \emph{single descent step}, which avoids projection and does not
require knowledge of the smoothness parameter \(L_{f_{S}}\).

\subsection{Weak-Separation Oracle}
\label{sec:weak-separ-oracle}

\begin{oracle}
  \caption{Weak-Separation Oracle \(\LPsep{P}(c, x, \Phi, K)\)}
  \label{alg:weak-separate-oracle}
  \begin{algorithmic}
    \REQUIRE
    linear objective \(c \in \mathbb{R}^{n}\),
    point \(x \in P\),
    accuracy \(K \geq 1\),
    gap estimate \(\Phi > 0\);
    \ENSURE Either
    \begin{enumerate*}
    \item\label{item:positive}
      vertex \(y \in P\) with
      \(c (x - y) \geq \Phi / K\)
    \item\label{item:negative}
      \FALSE: \(c (x - z) \leq \Phi\) for all \(z \in P\)
    \end{enumerate*}
  \end{algorithmic}
\end{oracle}

The weak-separation oracle (Oracle~\ref{alg:weak-separate-oracle}) was
introduced in \cite{braun2016lazifying} to replace the LP oracle
traditionally used in the CG method.  Provided with a point \(x \in
P\), a linear objective \(c\), a target reduction value $\Phi>0$, and
an inexactness factor $K \ge 1$, it decides whether there exists \(y
\in P\) with \(cx - cy \geq \Phi/K\), or else certifies that \(cx - cz
\leq \Phi\) for all \(z \in P\). In our applications, \(c=\nabla
f(x)\) is the gradient of the objective at the current iterate \(x\).
Oracle~\ref{alg:weak-separate-oracle} could be implemented simply by
the standard LP oracle of minimizing $cz$ over $z \in P$.  However, it
allows more efficient implementations, including the following.
\begin{enumerate*}[itemjoin*={{, and }}]
\item \emph{Caching}: testing previously obtained vertices $y \in P$
  (specifically, vertices in the current
  active vertex set) to see if one of them satisfies \(cx - cy \geq
  \Phi/K\). If not, the traditional LP oracle could be called to
  either find a new vertex of $P$ satisfying this bound, or else to
  certify that \(cx - cz \leq \Phi\) for all \(z \in P\)
\item \emph{Early Termination}: Terminating the LP procedure as soon
  as a vertex of $P$ has been discovered that satisfies \(cx - cy \geq
  \Phi/K\). (This technique requires an LP implementation that
  generates vertices as iterates.) If the LP procedure runs to
  termination without finding such a point, it has certified that \(cx
  - cz \leq \Phi\) for all \(z \in P\)
\end{enumerate*}
In \cite{braun2016lazifying}
these techniques resulted in orders-of-magnitude speedups in
wall-clock time in the computational tests,
as well as sparse convex combinations of
vertices for the iterates \(x_t\), a desirable property in many
contexts.

\section{Blended Conditional Gradients}
\label{sec:offline-pgd-lazy}
Our BCG approach is specified as Algorithm~\ref{alg:LOLCG}. We
discuss the algorithm in this section and establish its convergence
rate.  The algorithm expresses each iterate $x_t$, $t=0,1,2,\dotsc$ as
a convex combination of the elements of the active vertex set, denoted
by $S_t$, as in the Pairwise and Away-step variants of CG. At each
iteration, the algorithm calls either Oracle~\ref{ora:simplex} or
Oracle~\ref{alg:weak-separate-oracle} in search of the next iterate,
whichever promises the smaller function value, using a test in
Line~\ref{line:pgdIf} based on an estimate of the dual gap. The
same greedy principle is used in the Away-step CG approach, and its
lazy variants.
A critical role in the algorithm (and particularly in the test of
Line~\ref{line:pgdIf}) is played by the value $\Phi_t$, which is a
current estimate of the primal gap --- the difference between the
current function value $f(x_t)$ and the optimal function value over
$P$.  When Oracle~\ref{alg:weak-separate-oracle} returns \textbf{false},
the curent value of $\Phi_t$ is discovered to be an overestimate of
the dual gap, so it is halved
(Line~\ref{line:halve-Phi}) and we proceed to the next iteration.
In subsequent discussion, we refer to $\Phi_t$ as the ``gap estimate.''

\begin{algorithm}
  \caption{Blended Conditional Gradients (BCG)}
  \label{alg:LOLCG}
  \begin{algorithmic}[1]
    \REQUIRE
      smooth convex function \(f\),
      start vertex \(x_{0} \in P\),
      weak-separation oracle \(\LPsep{P}\),
      accuracy \(K \geq 1\)
    \ENSURE points \(x_{t}\) in \(P\) for \(t=1, \dots, T\)
    \STATE \(\Phi_{0} \leftarrow
      \max_{v \in P} \nabla f(x_{0})(x_{0} - v) / 2\)
      \COMMENT{Initial gap estimate}
    \STATE \(S_{0} \leftarrow \{x_{0}\}\)
    \FOR{\(t=0\) \TO \(T-1\)}
      \STATE \(v^{A}_{t} \leftarrow
        \argmax_{v \in S_{t}} \nabla f(x_{t}) v\)
        \label{line:away-atom}
      \STATE \(v^{FW-S}_{t} \leftarrow
        \argmin_{v \in S_{t}} \nabla f(x_{t}) v\)
        \label{line:FW-in-S-atom}
      \IF{\(\nabla f(x_{t}) (v^{A}_{t} - v^{FW-S}_{t})
            \geq \Phi_{t}\)} \label{line:pgdIf}
        \STATE \(x_{t+1}, S_{t+1} \leftarrow \Sora(x_{t}, S_{t})\)
          \label{line:Sora} \COMMENT{either a drop step or a descent step}
        \STATE \(\Phi_{t+1} \leftarrow \Phi_{t}\)
      \ELSE
        \STATE \(v_{t} \leftarrow
          \LPsep{P}(\nabla f(x_{t}), x_{t}, \Phi_{t}, K)\)
        \IF{\(v_{t} = \FALSE\)}
          \STATE \label{line:LPNoneUpdate}\(x_{t+1} \leftarrow x_{t}\)
          \STATE \(\Phi_{t+1} \leftarrow \Phi_{t} / 2\)
            \label{line:halve-Phi}
            \COMMENT{gap step}
          \STATE \(S_{t+1} \leftarrow S_{t}\)
        \ELSE
          \STATE \(x_{t+1} \leftarrow
            \argmin_{x \in [x_{t}, v_{t}]} f(x)\)
            \label{line:line-search}
            \COMMENT{FW step, with line search}
          \STATE Choose \(S_{t+1} \subseteq S_{t} \cup \{v_{t}\}\)
            minimal such that \(x_{t+1} \in \conv S_{t+1}\).
            \label{line:minimal-S-FW}
          \STATE \(\Phi_{t+1} \leftarrow \Phi_{t}\)
        \ENDIF
      \ENDIF
    \ENDFOR
  \end{algorithmic}
\end{algorithm}

  In Line~\ref{line:minimal-S-FW}, the active set \(S_{t+1}\) is
  required to be minimal. By Caratheodory's theorem, this requirement
  ensures that \(\size{S_{t+1}} \leq \dim P + 1\).  In practice, the
  \(S_{t}\) are invariably small and no explicit reduction in size is
  necessary.  The key requirement, in theory and practice, is that if
  after a call to \Sora the new iterate \(x_{t+1}\) lies on a
  face of the convex hull of the vertices in $S_t$, then at least one
  element of \(S_{t}\) is dropped to form $S_{t+1}$.  This requirement
  ensures that the local pairwise gap in Line~\ref{line:pgdIf} is not
  too large due to stale vertices in \(S_{t}\), which can block
  progress.  Small size of the sets \(S_{t}\) is crucial to the
  efficiency of the algorithm, in rapidly determining the maximizer
  and minimizer of $\nabla f(x_t)$ over the active set \(S_{t}\)
  in Lines~\ref{line:away-atom} and~\ref{line:FW-in-S-atom}.

The constants in the convergence rate described in our main theorem
(Theorem~\ref{thm:pgd-FW} below) depend on a modified curvature-like
parameter of the function \(f\).  Given a vertex set \(S\) of \(P\),
recall from Section~\ref{sec:proj-grad-desc} the smoothness parameter
\(L_{f_{S}}\) of the function \(f_{S} \colon \Delta^k \to \R \)
defined by \eqref{eq:def.fS}.  Define the \emph{simplicial curvature}
$C^\Delta$ to be
\begin{equation} \label{eq:CDel}
  C^{\Delta} \coloneqq \max_{S
    \colon \size{S} \leq 2 \dim P} L_{f_{S}}
\end{equation}
to be the maximum of the
\(L_{f_{S}}\) over all possible active sets.  This  affine-invariant parameter
depends both on the shape of \(P\) and the function \(f\).  This is the relative
smoothness constant \(L_{f,A}\) from
the predecessor of \cite{condition2019},
namely \cite[Definiton~2a]{condition2018},
with an additional restriction: the
simplex is restricted to faces of dimension at most \(2 \dim P\),
which appears as a bound on the size of \(S\) in our formulation.
This restriction improves the constant by removing dependence on the
number of vertices of the polytope, and can probably replace the
original constant in convergence bounds.  We can immediately see the
effect in the common case of \(L\)-smooth functions, that the
simplicial curvature is of reasonable magnitude, specifically,
\begin{equation}
  \label{eq:upper-bound-simplicial-curvature}
  C^{\Delta} \leq  \frac{L D^{2} (\dim P)}{2}
  ,
\end{equation}
where \(D\) is the diameter of \(P\). This result follows from
\eqref{eq:CDel} and
the bound on \(L_{f_{S}}\) from
Lemma~\ref{lem:upper-bound-simplex-smoothness} in the appendix.
This
bound is not directly comparable with the upper bound \(L_{f, A}
\leq L D^{2} / 4\) in \cite[Corollary~2]{condition2018}, because the
latter uses the \(1\)-norm on the probability simplex, while we use the
\(2\)-norm, the norm used by projected gradients and our simplex
gradient descent.  The additional factor \(\dim P\) is explained by
the \(n\)-dimensional probability simplex having
constant minimum width \(2\) in \(1\)-norm,
but having minimum width dependent on the dimension $n$
(specifically,  \(\Theta(1/\sqrt{n})\)) in the \(2\)-norm.
Recall that the minimum width of a convex body \(P \subseteq \R^{n}\)
in norm \(\norm{\cdot}\) is \(\min_{\phi} \max_{u, v \in P} \phi(u-v)\),
where \(\phi\) runs over all linear maps \(\R^{n} \to \R\)
having dual norm \(\dualnorm{\phi} = 1\).
For the \(2\)-norm, this is just the minimum distance between parallel
hyperplanes such that \(P\) lies between the two hyperplanes.

For another comparison, recall the curvature bound \(C \leq L D^{2}\).
Note, however, that the algorithm and convergence rate below are affine
invariant, and the only restriction on the function \(f\) is that it
has finite simplicial curvature.
This restriction readily provides the curvature bound
\begin{equation} \label{eq:CCD}
  C  \leq 2 C^{\Delta},
\end{equation}
where the factor \(2\) arises as the square
of the diameter of the probability simplex \(\Delta^{k}\).
(See Lemma~\ref{lem:curvature-by-simplicial} in the appendix for details.)
Note that \(S\) is allowed to be large enough
so that every two points of \(P\) lie simultaneously
in the convex hull of some vertex subset \(S\), by
Caratheodory's theorem, which is needed for \eqref{eq:CCD}.

We describe the convergence of BCG (Algorithm~\ref{alg:LOLCG}) in the
following theorem.

\begin{theorem}
  \label{thm:pgd-FW}
  Let \(f\) be a strongly convex, smooth function
  over the polytope \(P\) with simplicial curvature \(C^{\Delta}\) and geometric strong
  convexity \(\mu\). Then
  Algorithm~\ref{alg:LOLCG} ensures
  \(f(x_{T}) - f(x^{*}) \leq \varepsilon\),
  where \(x^{*}\)
  is an optimal solution to \(f\)
  in \(P\) for some iteration index $T$ that satisfies
  \begin{equation}
    \label{eq:pgd-FW}
    T
    \leq
    \left\lceil \log \frac{2 \Phi_0}{\varepsilon} \right\rceil
    + 8 K
    \left\lceil \log \frac{\Phi_0}{2K C^\Delta} \right\rceil
    +
    \frac{64 K^2 C^\Delta}{\mu}
    \left\lceil \log \frac{4K C^\Delta}{\varepsilon} \right\rceil
    =
    O \left( \frac{C^\Delta}{\mu} \log \frac{\Phi_0}{\varepsilon}
    \right)
    ,
  \end{equation}
  where $\log$ denotes logarithms to the base $2$.
\end{theorem}
For smooth but not necessarily strongly convex functions \(f\), the
algorithm ensures \(f(x_{T}) - f(x^{*}) \leq
\varepsilon\) after \(T=O(\max\{ C^{\Delta}, \Phi_{0} \} /
\varepsilon)\) iterations by a similar argument, which is omitted.
\begin{proof}
The proof tracks that of \cite{braun2016lazifying}.  We divide the
iteration sequence into epochs that are demarcated by the
\emph{gap steps}, that is, the iterations for which
the weak-separation oracle (Oracle~\ref{alg:weak-separate-oracle})
returns the value \textbf{false}, which results in $\Phi_t$ being halved
for the next iteration.
We then bound the number of iterates within each epoch.
The result is obtained by aggregating
across epochs.

We start by a well-known bound on the function value using the
Frank–Wolfe point
\begin{equation} \label{eq:et2}
  v^{FW}_t \coloneqq \argmin_{v \in P} \nabla f(x_t)v
\end{equation}
at iteration $t$, which follows from convexity:
\[
f(x_t) - f(x^*) \le \nabla f(x_t) (x_t-x^*) \le \nabla f(x_t) (x_t-v^{FW}_t).
\]
If iteration $t-1$ is a gap step, we have
using $x_t=x_{t-1}$ and $\Phi_t = \Phi_{t-1}/2$ that
\begin{equation} \label{eq:et3}
  f(x_t) - f(x^*) \le \nabla f(x_t) (x_t-v^{FW}_t) \le 2 \Phi_t.
\end{equation}
This bound also holds at $t=0$, by definition of $\Phi_0$. Thus
Algorithm~\ref{alg:LOLCG} is guaranteed to satisfy $f(x_T)-f(x^*) \le
\varepsilon$ at some iterate $T$ such that $T-1$ is a gap
step and $2 \Phi_T \le \varepsilon$. Therefore, the total number
of gap steps $\NPhi$ required to reach this
point satisfies
\begin{equation} \label{eq:Nneg}
 \NPhi \leq
  \left\lceil \log \frac{2\Phi_0}{\varepsilon} \right\rceil,
\end{equation}
which is also a bound on the total number of epochs. The next stage of
the proof finds bounds on the number of iterations of each type within
an individual epoch.

If iteration $t-1$ is a gap step, we have
$x_t=x_{t-1}$ and $\Phi_t = \Phi_{t-1}/2$, and because
the condition is false at Line \ref{line:pgdIf} of
Algorithm~\ref{alg:LOLCG}, we have
\begin{equation} \label{eq:et5}
  \nabla f(x_t)(v^A_t-x_t) \le \nabla f(x_t) (v^A_t - v^{FW-S}_t) \le 2 \Phi_t.
\end{equation}
This condition also holds trivially at $t=0$, since $v^A_0 =
v^{FW-S}_0=x_0$. By summing \eqref{eq:et3} and \eqref{eq:et5}, we obtain
\[
\nabla f(x_t)(v^A_t-v^{FW}_t) \leq 4 \Phi_t,
\]
so it follows from Fact~\ref{fact:strongConvP} that
\[
f(x_t)-f(x^*) \le \frac{[\nabla f(x_t)(v^A_t-v^{FW}_t)]^2}{2 \mu} \le \frac{8 \Phi_t^2}{\mu}.
\]
By combining this inequality with \eqref{eq:et3}, we obtain
\begin{equation} \label{eq:et7}
  f(x_t)-f(x^*) \le \min \left\{ 8 \Phi_t^2/\mu, 2 \Phi_t \right\}
\end{equation}
for all $t$ such that either $t=0$ or else $t-1$ is a gap
step.  In fact, \eqref{eq:et7} holds for \emph{all} $t$, because
\begin{enumerate*}[itemjoin={{; }}, itemjoin*={{; and }}]
\item
  the sequence of function values $\{f(x_s)\}_s$ is non-increasing
\item $\Phi_s = \Phi_t$ for all $s$ in the epoch that starts at iteration $t$
\end{enumerate*}

We now consider the epoch that starts at iteration $t$, and use $s$ to
index the iterations within this epoch. Note that $\Phi_s =
\Phi_t$ for all $s$ in this epoch.

We distinguish three types of iterations besides gap
step.  The first type is a \emph{Frank–Wolfe} step,
taken when the weak-separation oracle returns an
improving vertex $v_s \in P$ such that $\nabla f(x_s)(x_s-v_s) \ge
\Phi_s/K = \Phi_t/K$ (Line \ref{line:line-search}).  Using the
definition of curvature $C$, we have by standard Frank–Wolfe arguments
that
(c.f., \cite{braun2016lazifying}).
\begin{equation}
  \label{eq:FW-step}
  f(x_s) - f(x_{s+1})
  \geq
  \frac{\Phi_s}{2K} \min \left\{ 1, \frac{\Phi_s}{KC} \right\}
  \geq
  \frac{\Phi_t}{2K} \min \left\{ 1, \frac{\Phi_t}{2KC^\Delta} \right\}
  ,
\end{equation}
where we used $\Phi_s=\Phi_t$ and $C \le 2C^\Delta$ (from
\eqref{eq:CCD}).  We denote by $N_{\text{FW}}^t$ the number of
Frank–Wolfe iterations in the epoch starting at iteration
$t$.

The second type of iteration is a \emph{descent step}, in which
\Sora (Line \ref{line:Sora}) returns a point $x_{s+1}$ that lies in
the relative interior of $\conv S_s$ and with strictly smaller
function value.
We thus have $S_{s+1} = S_s$ and, by the definition of \Sora, together with
\eqref{eq:CDel},
it follows that
\begin{equation}
  \label{eq:gradient-step}
  f(x_s) - f(x_{s+1})
  \geq
  \frac{[\nabla f(x_s) (v^A_s - v^{FW-S}_s)]^2}{4 C^\Delta}
  \geq
  \frac{\Phi_s^2}{4C^\Delta}
  =
  \frac{\Phi_t^2}{4C^\Delta}
  .
\end{equation}
We denote by $\Nd^t$ the number of descent steps
that take place in the epoch that starts at iteration $t$.

The third type of iteration is one in which Oracle~\ref{ora:simplex}
returns a point $x_{s+1}$ lying on a face of the convex hull of
$S_s$, so that $S_{s+1}$ is strictly smaller than $S_s$. Similarly to
the Away-step Frank–Wolfe algorithm of \cite{FW-converge2015}, we call
these steps \emph{drop steps}, and denote by $N^t_{\text{drop}}$ the
number of such steps that take place in the epoch that starts at
iteration $t$.  Note that since $S_s$ is expanded only at Frank–Wolfe
steps, and then only by at most one element, the \emph{total} number
of drop steps across the whole algorithm cannot exceed the total
number of Frank–Wolfe steps.  We use this fact and \eqref{eq:Nneg} in
bounding the total number of iterations $T$ required for
$f(x_T)-f(x^*) \le \varepsilon$:
\begin{equation}
  \label{eq:rounds}
  T
  \leq
 \NPhi  + \Nd + N_{\text{FW}} + N_{\text{drop}}
  \leq
  \left\lceil
    \log \frac{2\Phi_0}{\varepsilon}
  \right\rceil
  + \Nd + 2 N_{\text{FW}}
  =
  \left\lceil
    \log \frac{2\Phi_0}{\varepsilon}
  \right\rceil
  + \sum_{t: \text{epoch start}}
  (\Nd^{t} + 2 N^{t}_{\text{FW}})
  .
\end{equation}
Here $\Nd$
denotes the total number of descent steps,
$N_{\text{FW}}$ the total number of Frank–Wolfe steps, and
$N_{\text{drop}}$ the total number of drop steps, which is
bounded by $N_{\text{FW}}$, as just discussed.

Next, we seek bounds on the iteration counts $\Nd^t$ and $N^t_{\text{FW}}$
within the epoch starting with iteration $t$.
For the total decrease in function value during the epoch,
Equations \eqref{eq:FW-step} and \eqref{eq:gradient-step}
provide a lower bound, while \(f(x_{t}) - f(x^{*})\) is an obvious
upper bound, leading to the following estimate using \eqref{eq:et7}.
\begin{itemize}
  \item
If $\Phi_t \ge 2KC^\Delta$ then
\begin{equation}
  \label{eq:progress-far-computation}
  2 \Phi_{t}
  \geq
  f(x_{t}) - f(x^{*})
  \geq
  \Nd^t \frac{\Phi_t^2}{4C^\Delta}
  +
  N^t_{\text{FW}} \frac{\Phi_t}{2K}
  \geq
  \Nd^t \frac{\Phi_t K}{2}
  +
  N^t_{\text{FW}} \frac{\Phi_t}{2K}
  \geq
  (\Nd^t + 2 N^t_{\text{FW}}) \frac{\Phi_t}{4 K}
  ,
\end{equation}
hence
\begin{equation}
  \label{eq:progress-far}
  \Nd^t + 2 N^t_{\text{FW}}
  \leq
  8 K
  .
\end{equation}

\item
  If $\Phi_t < 2K C^\Delta$, a similar argument
provides
\begin{equation}
  \label{eq:progress-near-computation}
  \frac{8 \Phi_{t}^{2}}{\mu}
  \geq
  f(x_{t}) - f(x^{*})
  \geq
  \Nd^t \frac{\Phi_t^2}{4C^\Delta}
  +
  N^t_{\text{FW}} \frac{\Phi_t^2}{4K^2 C^\Delta}
  \geq
  (\Nd^t + 2 N^t_{\text{FW}}) \frac{\Phi_t^2}{8 K^2 C^\Delta}
  ,
\end{equation}
leading to
\begin{equation}
  \label{eq:progress-near}
  \Nd^t + 2 N^t_{\text{FW}}
  \leq
  \frac{64 K^{2} C^{\Delta}}{\mu}
  .
\end{equation}
\end{itemize}
There are at most
\begin{align*}
  &
  \left\lceil \log \frac{\Phi_0}{2KC^\Delta} \right\rceil
  \text{epochs in the regime with \(\Phi_{t} \geq 2 K C^{\Delta}\)}, \\
  &
  \left\lceil \log \frac{2KC^\Delta}{\varepsilon/2} \right\rceil
  \text{epochs in the regime with \(\Phi_{t} < 2 K C^{\Delta}\)}.
\end{align*}

Combining \eqref{eq:rounds} with the bounds
\eqref{eq:progress-far} and \eqref{eq:progress-near}
on $N_{\text{FW}}^{t}$ and $\Nd^{t}$,
we obtain
\eqref{eq:pgd-FW}.
\end{proof}

\section{Simplex Gradient Descent}
\label{sec:proj-grad-desc-projfree}

Here we describe the Simplex Gradient Descent approach
(Algorithm~\ref{alg:simplex-descent}), an implementation of \Sora
(Oracle~\ref{ora:simplex}).
Algorithm~\ref{alg:simplex-descent} requires only $O(\size{S})$
operations beyond the evaluation of $\nabla f(x)$ and the cost of line
search. (It is assumed that \(x\) is represented as a convex
combination of vertices of $P$, which is updated during
Oracle~\ref{ora:simplex}.)  Apart from the (trivial) computation of
the projection of $\nabla f(x)$ onto the linear space spanned by
$\Delta^k$, no projections are computed. Thus,
Algorithm~\ref{alg:simplex-descent} is typically faster even than a
Frank–Wolfe step (LP oracle call), for typical small sets \(S\).

Alternative implementations of Oracle~\ref{ora:simplex} are described
in Section~\ref{sec:ora-alt}.  Section~\ref{sec:example-simplex}
describes the special case in which $P$ itself is a probability
simplex, combining BCG and its oracles into a single, simple
method with better constants in the convergence bounds.

In the algorithm,
the form \(c \allOne\) denotes the scalar product of \(c\) and
\(\allOne\), i.e., the sum of entries of \(c\).

\begin{algorithm}
  \caption{Simplex Gradient Descent Step (SiGD)}
  \label{alg:simplex-descent}
  \begin{algorithmic}[1]
    \REQUIRE
      polyhedron $P$, smooth convex function \(f \colon P \to \mathbb{R}\), subset $S = \{v_1,v_2,\dotsc,v_k \}$ of vertices of $P$,
      point \(x \in \conv S\)
    \ENSURE set \(S' \subseteq S\),
      point \(x' \in \conv S'\)
    \STATE Decompose \(x\) as a convex combination
      \(x = \sum_{i=1}^{k} \lambda_{i} v_{i}\), with $\sum_{i=1}^k \lambda_i=1$ and $\lambda_i \ge 0$, $i=1,2,\dotsc,k$
      \label{line:simplex-decompose}
    \STATE \(c \leftarrow
      [\nabla f(x) v_{1}, \dotsc, \nabla f(x) v_{k}]\)
      \COMMENT{\(c = \nabla f_{S}(\lambda)\); see \eqref{eq:def.fS}}
    \STATE \(d \leftarrow c - (c \allOne) \allOne / k\)
      \COMMENT{Projection onto the hyperplane of $\Delta^k$}
      \label{line:projected-gradient}
    \IF{\(d = 0\)}
      \RETURN \(x' = v_{1}\), \(S' = \{v_{1}\}\)
        \COMMENT{Arbitrary vertex}
    \ENDIF
    \STATE \(\eta \leftarrow
      \max \{\eta \geq 0 : \lambda - \eta d \geq 0\}\)
    \STATE \(y \leftarrow x - \eta \sum_{i} d_{i} v_{i}\)
      \label{line:simplex-boundary}
    \IF{\(f(x) \geq f(y)\)} \label{line:accept-drop}
      \STATE \(x' \leftarrow y\)
      \STATE Choose \(S' \subseteq S\), \(S' \neq S\) with
        \(x' \in \conv S'\).
        \label{line:drop}
    \ELSE
      \STATE \(x' \leftarrow \argmin_{z \in [x, y]} f(z)\)
        \label{line:simplex-line-search}
      \STATE \(S' \leftarrow S\)
    \ENDIF
    \RETURN \(x'\), \(S'\)
  \end{algorithmic}
\end{algorithm}

To verify the validity of Algorithm~\ref{alg:simplex-descent} as an
implementation of Oracle~\ref{ora:simplex},
note first that since \(y\)
lies on a face of \(\conv S\) by definition, it is always possible to
choose a proper subset \(S' \subseteq S\) in Line~\ref{line:drop}, for
example, \(S' \coloneqq \{v_{i} : \lambda_{i} > \eta d_{i}\}\).  The
following lemma shows that with the choice \(h \coloneqq f_{S}\),
Algorithm~\ref{alg:simplex-descent} correctly implements
Oracle~\ref{ora:simplex}.

\begin{lemma}
  \label{lem:simplex-descent-step}
  Let $\Delta^k$ be the probability simplex in $k$ dimensions and
  $h \colon \Delta^k \to \R$ be an $L_h$-smooth
  function. Given some $\lambda \in
  \Delta^k$, define $d \coloneqq \nabla h(\lambda) - (\nabla h(\lambda)
  \allOne/k) \allOne$ and let $\eta \ge 0$ be the largest value for
  which $\tau \coloneqq \lambda - \eta d \ge 0$. Let $\lambda' \coloneqq \argmin_{z
    \in [\lambda,\tau]} \, h(z)$. Then either $h(\lambda) \ge h(\tau)$
  or
  \[
    h(\lambda) - h(\lambda') \ge
    \frac{[\max_{1 \leq i,j \leq k} \nabla h(\lambda)(e_i-e_j)]^2}{4L_h}.
  \]
\end{lemma}
\begin{proof}
Let \(\pi\) denote the orthogonal projection onto
the lineality space of \(\Delta^{k}\),
i.e., \(\pi(\zeta) \coloneqq \zeta - (\zeta \allOne) \allOne / k\).
Let \(g(\zeta) \coloneqq h(\pi(\zeta))\),
then \(\nabla g(\zeta) = \pi (\nabla h(\pi(\zeta)))\),
and \(g\) is clearly \(L_{h}\)-smooth, too.
In particular, \(\nabla g(\lambda) = d\).

From standard gradient descent bounds, not repeated here,
we have the following
inequalities, for \(\gamma \leq \min\{\eta, 1 / L_{h}\}\):
\begin{equation}
  \label{eq:tk3}
 \begin{split}
  h(\lambda) - h(\lambda - \gamma d)
  &
  =
  g(\lambda) - g(\lambda - \gamma \nabla g(\lambda))
  \geq
  \gamma
  \frac{\norm[2]{\nabla g(\lambda)}^{2}}{2}
  \\
  &
  \geq
  \gamma
  \frac{[\max_{1 \leq i, j \leq k}
    \nabla g(\lambda) (e_{i} - e_{j})]^{2}}{4}
  =
  \gamma
  \frac{[\max_{1 \leq i, j \leq k}
    \nabla h(\lambda) (e_{i} - e_{j})]^{2}}{4}
  ,
 \end{split}
\end{equation}
where the second inequality uses that the \(\ell_{2}\)-diameter of the
\(\Delta^{k}\) is \(2\), and the last equality follows from
\(\nabla g(\lambda) (e_{i} - e_{j})
= \nabla h(\lambda) (e_{i} - e_{j})\).

When $\eta \ge 1/L_h$, we conclude that
$h(\lambda') \le h(\lambda - (1/L_h)d) \leq h(\lambda)$, hence
  \[
    h(\lambda)-h(\lambda') \ge \frac{[\max_{i,j \in \{1,2,\dotsc,k\}}
      \nabla h(\lambda)(e_i-e_j)]^2}{4 L_h},
  \]
  which is the second case of the lemma.
  When $\eta < 1/L_h$, then
  setting $\gamma = \eta$ in \eqref{eq:tk3} clearly provides
  \(h(\lambda) - h(\tau) \geq 0\),
  which is the first case of the lemma.
\end{proof}

\subsection{Alternative implementations of Oracle~\ref{ora:simplex}}
\label{sec:ora-alt}

Algorithm~\ref{alg:simplex-descent} is probably the least expensive
possible implementation of Oracle~\ref{ora:simplex}, in general. We
may consider other implementations, based on
projected gradient descent, that aim to decrease $f$ by a greater amount in
each step and possibly make more extensive reductions to the set
$S$. \emph{Projected gradient descent} would seek to minimize $f_S$ along the piecewise-linear path
\(\{\operatorname{proj}_{\Delta^k}(\lambda - \gamma \nabla
f_S(\lambda)) \mid \gamma \geq 0 \}\),
where \(\operatorname{proj}_{\Delta^k}\) denotes projection
onto \(\Delta^{k}\).
Such a search is more expensive, but may
  result in a new active set $S'$ that is significantly smaller
  than the current set $S$ and, since the reduction in $f_S$ is at
  least as great as the reduction on the interval $\gamma \in
  [0,\eta]$ alone, it also satisfies the requirements of
  Oracle~\ref{ora:simplex}.

  More advanced methods for optimizing over the simplex could also be
  considered, for example, mirror descent
  (see \cite{nemirovskii1983problem}) and accelerated versions of mirror
  descent and projected gradient descent; see \cite{Lan2017-MLnotes} for a
  good overview. The effects of these alternatives on the overall
  convergence rate of Algorithm~\ref{alg:LOLCG} has not been studied;
  the analysis is complicated significantly by the lack of guaranteed improvement in
  each (inner) iteration.

  The accelerated versions are considered in the computational tests
  in Section~\ref{sec:comp-results}, but on the examples we
  tried, the inexpensive implementation of
  Algorithm~\ref{alg:simplex-descent} usually gave the fastest overall
  performance.
  We have not tested mirror descent versions.

\subsection{Simplex Gradient Descent as a stand-alone algorithm}
\label{sec:example-simplex}

We describe a variant of Algorithm~\ref{alg:LOLCG} for the special
case in which $P$ is the probability simplex \(\Delta^{k}\). Since
optimization of a linear function over $\Delta^k$ is trivial, we use
the standard LP oracle in place of the weak-separation oracle
(Oracle~\ref{alg:weak-separate-oracle}), resulting in the non-lazy
variant Algorithm~\ref{alg:sigd}.  Observe that the per-iteration cost
is only \(O(k)\). In cases of \(k\) very large, we could also
formulate a version of Algorithm~\ref{alg:sigd} that uses a
weak-separation oracle (Oracle~\ref{alg:weak-separate-oracle}) to
evaluate only a subset of the coordinates of the gradient, as in
coordinate descent. The resulting algorithm would be an interpolation
of Algorithm~\ref{alg:sigd} below and Algorithm~\ref{alg:LOLCG};
details are left to the reader.

\begin{algorithm}
  \caption{Stand-Alone Simplex Gradient Descent}
  \label{alg:sigd}
  \begin{algorithmic}[1]
    \REQUIRE
      convex function \(f\)
    \ENSURE points \(x_{t}\) in \(\Delta^{k}\) for \(t=1, \dots, T\)
    \STATE \(x_{0} = e_{1}\)
    \FOR{\(t=0\) \TO \(T-1\)}
      \STATE \(S_t \leftarrow \{i \colon x_{t, i} > 0\}\)
      \STATE \(a_t \leftarrow \argmax_{i \in S_t}
        \nabla f(x_t)_{i}\)
      \STATE \(s_t \leftarrow \argmin_{i \in S_t}
        \nabla f(x_t)_{i}\)
      \STATE \(w_t \leftarrow \argmin_{1 \leq i \leq k} \nabla f(x_t)_{i}\)
      \IF{\(\nabla f(x_{t})_{a_t} - \nabla f(x_{t})_{s_t} >
        \nabla f(x_{t}) x_{t} - \nabla f(x_{t})_{w_t}\)}
      \label{line:condition}
        \STATE \(d_{i} =
          \begin{cases}
            \nabla f(x_t)_{i}
            - \sum_{j \in S} \nabla f(x_t)_{j} / \size{S_t}
            & i \in S_t \\
            0 & i \notin S_t
          \end{cases}\)
          \qquad
          for \(i=1, 2, \dotsc, k\)
        \STATE \(\eta = \max \{\gamma : x_{t} - \gamma d \geq 0\}\)
        \COMMENT{ratio test}
        \STATE \(y = x_{t} - \eta d\)
        \IF{\(f(x_{t}) \geq f(y)\)}
          \STATE \(x_{t+1} \leftarrow y\)
        \COMMENT{drop step}
            \label{line:sigd-drop}
        \ELSE
          \STATE \(x_{t+1} \leftarrow
           \argmin_{x \in [x_{t}, y]} f(x)\)
            \label{line:sigd-gd}
           \COMMENT{descent step}
        \ENDIF
      \ELSE
        \STATE \(x_{t+1} \leftarrow \argmin_{x \in[x, e_{w_t}]} f(x)\)
          \label{line:sigd-FW}
        \COMMENT{FW step}
      \ENDIF
    \ENDFOR
  \end{algorithmic}
\end{algorithm}

When line search is too expensive, one might replace
Line~\ref{line:sigd-gd} by \(x_{t+1} = (1 - 1 / L_{f}) x_{t} + y /
L_{f}\), and Line~\ref{line:sigd-FW} by \(x_{t+1} = (1 - 2 / (t + 2))
x_{t} + (2 / (t + 2)) e_{w}\).  These employ the standard step sizes
for the Frank–Wolfe algorithm and (projected) gradient descent, respectively, and
yield the required descent guarantees.

We now describe convergence rates for Algorithm~\ref{alg:sigd}, noting
that better constants are available in the convergence rate expression
than those obtained from a direct application of
Theorem~\ref{thm:pgd-FW}.

\begin{corollary}
  \label{cor:sigd}
  Let \(f\) be an \(\alpha\)-strongly convex and \(L_{f}\)-smooth
  function over the probability simplex \(\Delta^{k}\) with $k \ge 2$.
  Let \(x^{*}\) be a minimum point of \(f\) in \(\Delta^{k}\).  Then
  Algorithm~\ref{alg:sigd} converges with rate
  \begin{equation}
    \label{eq:sigd}
    f(x_{T}) - f(x^{*})
    \leq
    \left(
      1 - \frac{\alpha}{4 L_{f} k}
    \right)^{T}
    \cdot
    \left(f(x_{0}) - f(x^{*})\right),
    \quad
    T = 1,2,\dotsc.
  \end{equation}
  If \(f\) is not strongly convex (that is, \(\alpha = 0\)), we have
  \begin{equation}
    \label{eq:sigd-smooth}
    f(x_{T}) - f(x^{*})
    \leq
    \frac{8 L_{f}}{T}, \quad T=1,2,\dotsc.
  \end{equation}
\begin{proof}
  The structure of the proof is similar to that of
  \cite[Theorem~8]{FW-converge2015}. Recall from
  \cite[§B.1]{FW-converge2015} that the pyramidal width of the
  probability simplex is \(W \geq 2 / \sqrt{k}\), so that the geometric
  strong convexity of \(f\) is \(\mu \geq 4 \alpha / k\).  The
  diameter of \(\Delta^{k}\) is \(D = \sqrt{2}\), and it is easily
  seen that \(C^{\Delta} = L_{f}\) and \(C \leq L_{f} D^{2} / 2 =
  L_{f}\).

To maintain the same notation as in the proof of
Theorem~\ref{thm:pgd-FW}, we define \(v_t^{A} = e_{a_t}\),
\(v^{FW-S}_t = e_{s_t}\) and \(v^{FW}_t = e_{w_t}\). In particular, we
have \(\nabla f(x_{t})_{w_t} = \nabla f(x_{t}) v^{FW}_t\),
\(\nabla f(x_{t})_{s_t} = \nabla f(x_{t}) v^{FW-S}_t\), and \(\nabla
f(x_{t})_{a_t} = \nabla f(x_{t}) v^{A}_t\).  Let \(h_{t} \coloneqq
f(x_{t}) - f(x^{*})\).

In the proof, we use several elementary estimates. First, by convexity
of $f$ and the definition of the Frank–Wolfe step, we have
\begin{equation} \label{eq:hz6}
  h_t = f(x_t)-f(x^*) \le \nabla f(x_t)(x_t-v^{FW}_t).
\end{equation}
Second, by Fact~\ref{fact:strongConvP}
and the estimate $\mu \geq 4 \alpha / k$ for
geometric strong convexity, we obtain
\begin{equation} \label{eq:hz8}
  h_t \le \frac{[\nabla f(x_t)(v^A_t-v^{FW}_t)]^2}{8 \alpha/k}.
  \end{equation}

Let us consider a fixed iteration \(t\). Suppose first that we take a
descent step (Line~\ref{line:sigd-gd}),
in particular, \(\nabla f(x_{t}) (v^{A}_{t} - v^{FW-S}_{t}) \geq
\nabla f(x_{t}) (x_{t} - v^{FW}_{t})\) from Line~\ref{line:condition} which,
together with \(\nabla f(x_{t}) x_{t} \geq \nabla f(x_{t}) v^{FW-S}\), yields
\begin{equation}
  \label{eq:hz7}
  2 \nabla f(x_{t}) (v^{A}_{t} - v^{FW-S}_{t}) \geq
  \nabla f(x_{t}) (v^{A}_{t} - v^{FW}_{t})
  .
\end{equation}
By Lemma~\ref{lem:simplex-descent-step}, we have
\begin{equation}
  \label{eq:sigd-gd}
  f(x_{t}) - f(x_{t+1})
  \geq
  \frac{\left[
      \nabla f(x_{t}) (v^{A} - v^{FW-S})
    \right]^{2}
  }{4 L_{f}}
  \geq
  \frac{\left[
      \nabla f(x_{t}) (v^{A} - v^{FW})
    \right]^{2}
  }{16 L_{f}}
  \geq
  \frac{\alpha}{2 L_{f} k} \cdot h_{t},
\end{equation}
where the second inequality follows from \eqref{eq:hz7} and the third
inequality follows from \eqref{eq:hz8}.

If a Frank–Wolfe step is taken (Line~\ref{line:sigd-FW}), we have
similarly to \eqref{eq:FW-step}
\begin{equation}
  \label{eq:sigd-FW-step}
  f(x_{t}) - f(x_{t+1})
  \geq
  \frac{\nabla f(x_{t}) (x_{t} - v^{FW})}{2}
  \min \left\{
    1, \frac{\nabla f(x_{t}) (x_{t} - v^{FW})}{2 L_{f}}
  \right\}
  .
\end{equation}
Combining with \eqref{eq:hz6}, we have either
$f(x_t)-f(x_{t+1}) \ge h_t/2$ or
\begin{equation}
  \label{eq:2}
  f(x_{t}) - f(x_{t+1})
  \geq
  \frac{[\nabla f(x_{t}) (x_{t} -  v^{FW})]^{2}}{4 L_{f}}
  \geq
  \frac{\left[
      \nabla f(x_{t}) (v^{A} - v^{FW})
    \right]^{2}
  }{16 L_{f}}
  \geq
  \frac{\alpha}{2 L_{f} k} \cdot h_{t}
  .
\end{equation}
Since \(\alpha \leq L_{f}\), the latter is always smaller than the
former, and hence is a lower bound that holds for all Frank–Wolfe steps.

Since \(f(x_{t}) - f(x_{t+1}) = h_{t} - h_{t+1}\), we have
\(h_{t+1} \leq (1 - \alpha / (2 L_{f} k)) h_{t}\) for descent steps
and Frank–Wolfe steps, while obviously \(h_{t+1} \leq h_{t}\)
for drop steps (Line~\ref{line:sigd-drop}). For any given iteration
counter $T$, let $\Td$ be the number of descent steps taken
before iteration $T$, $T_{\text{FW}}$ be the number of Frank–Wolfe
steps taken  before iteration $T$, and $T_{\text{drop}}$ be the number
of drop steps taken  before iteration $T$. We have $T_{\text{drop}} \le
T_{\text{FW}}$, so that similarly to \eqref{eq:rounds}
\begin{equation} \label{eq:hz9}
  T = \Td + T_{\text{FW}} + T_{\text{drop}}
  \leq  \Td + 2 T_{\text{FW}}
  .
\end{equation}
By compounding the decrease at each iteration, and using
\eqref{eq:hz9} together with the identity $(1-\epsilon/2)^2 \ge
(1-\epsilon)$ for any $\epsilon \in (0,1)$, we have
\begin{equation*}
  \label{eq:9}
  h_{T}
  \leq
  \left(
    1 - \frac{\alpha}{2 L_{f} k}
  \right)^{\Td + T_{\text{FW}}}
  h_{0}
  \leq
  \left(
    1 - \frac{\alpha}{2 L_{f} k}
  \right)^{T / 2}
  h_{0}
  \leq
  \left(
    1 - \frac{\alpha}{4 L_{f} k}
  \right)^{T}
  \cdot
  h_{0}
  .
\end{equation*}

The case for the smooth but not strongly convex functions is similar:
we obtain for descent steps
\begin{equation}
  \label{eq:sigd-smooth-gd}
  h_{t} - h_{t+1}
  =
  f(x_{t}) - f(x_{t+1})
  \geq
  \frac{\left[
      \nabla f(x_{t}) (v^{A} - v^{FW-S})
    \right]^{2}
  }{4 L_{f}}
  \geq
  \frac{\left[
      \nabla f(x_{t}) (x - v^{FW})
    \right]^{2}
  }{4 L_{f}}
  \geq
  \frac{h_{t}^{2}}{4 L_{f}},
\end{equation}
where the second inequality follows from \eqref{eq:hz6}.

For Frank–Wolfe steps, we have by standard estimations
\begin{equation}
  \label{eq:sigd-FW}
  h_{t+1} \le
  \begin{cases}
    h_t -h_t^2/(4L_f) & \text{if $h_t \le 2L_f$,} \\
    L_{f} \leq h_{t} / 2 & \text{otherwise.}
  \end{cases}
\end{equation}

Given an iteration $T$, we define $T_{\text{drop}}$, $T_{\text{FW}}$
and $\Td$ as above, and show by induction that
\begin{equation} \label{eq:ht8}
  h_T \le \frac{4L_f}{\Td +T_{\text{FW}}}
  \quad
  \text{for $T \ge 1$.}
\end{equation}
Equation \eqref{eq:ht8}, i.e., \(h_{T} \leq 8 L_{f} / T\)
easily follows from this via \(T_{\text{drop}} \leq T_{\text{FW}}\).
Note that the first step is necessarily a Frank–Wolfe step,
hence the denominator is never \(0\).

If iteration \(T\) is a drop step, then \(T > 1\),
and the claim is obvious by induction from \(h_{T} \geq h_{T - 1}\).
Hence we assume that iteration \(T\) is either a descent step or a
Frank–Wolfe step.
If \(\Td + T_{\text{FW}} \leq 2\) then by
\eqref{eq:sigd-smooth-gd} or \eqref{eq:sigd-FW} we obtain either
\(h_{T} \leq L_{f} < 2 L_{f}\)
or \(h_{T} \leq h_{T-1} - h_{T-1}^{2} / (4 L_{f}) \leq 2 L_{f}\),
without using any upper bound on \(h_{T-1}\), proving \eqref{eq:ht8}
in this case.
Note that this includes the case \(T=1\),
the start of the induction.

Finally, if \(\Td + T_{\text{FW}} \geq 3\),
then \(h_{T-1} \leq 4 L_{f} / (\Td + T_{\text{FW}} - 1) \leq
2 L_{f}\) by induction,
therefore a familiar argument
using \eqref{eq:sigd-smooth-gd} or \eqref{eq:sigd-FW} provides
\[
h_{T} \le \frac{4L_f}{\Td+T_{\text{FW}} - 1}
-\frac{4L_f}{(\Td+T_{\text{FW}} - 1)^2}
 \le
\frac{4L_f}{\Td+T_{\text{FW}}},
\]
proving \eqref{eq:ht8} in this case, too,
finishing the proof.
\end{proof}
\end{corollary}

\section{Algorithmic enhancements}
\label{sec:impr-actu-impl}

We describe various enhancements that can be made to the BCG
algorithm, to improve its practical performance while staying broadly
within the framework above.
Computational testing with these enhancements is reported in
Section~\ref{sec:comp-results}.

\subsection{Sparsity and culling of active sets}
\label{sec:cull}

Sparse solutions (which in the current context means
``solutions that are a convex combination of a small number of
vertices of $P$'') are desirable for many applications. Techniques for
promoting sparse solutions in conditional gradients were considered in
\cite{rao2015forward}.  In many situations, a sparse approximate
solution can be identified at the cost of some increase in the value
of the objective function.

We explored two sparsification approaches, which can be applied
separately or together, and performed preliminary computational tests
for a few of our experiments in Section~\ref{sec:comp-results}.

\begin{enumerate}
\item \label{item:promoteDrop} \emph{Promoting drop steps.} Here we
  relax Line~\ref{line:accept-drop} in
  Algorithm~\ref{alg:simplex-descent} from testing \(f(y) \geq f(x)\)
  to \(f(y) \geq f(x) - \varepsilon\), where \(\varepsilon \coloneqq \min\{
  \frac{\max\{p,0\}}{2}, \varepsilon_0\}\) with \(\varepsilon_0 \in
  \R\) some upper bound on the accepted potential increase in
  objective function value and \(p\) being the amount of reduction in
  $f$ achieved on the latest iteration. This technique allows a
  controlled increase of the objective function value in return for
  additional sparsity. The same convergence analysis will apply, with
  an additional factor of $2$ in the estimates of the total number of
  iterations.
\item \label{item:postOpt} \emph{Post-optimization.} Once the
  considered algorithm has stopped with active set \(S_0\), solution
  \(x_0\), and dual gap \(d_0\), we re-run the algorithm with the same
  objective function \(f\) over \(\conv{S_0}\), i.e., we
  solve \(\min_{x \in \conv{S_0}} f(x)\) terminating when the dual gap
  reaches \(d_0\).
\end{enumerate}

These approaches can sparsify the solutions of the baseline algorithms
Away-step Frank–Wolfe, Pairwise Frank–Wolfe, and lazy Pairwise
Frank–Wolfe; see \cite{rao2015forward}. We observed, however, that the
iterates generated by BCG are often quite sparse. In fact, the
solutions produced by BCG are sparser than those produced by the
baseline algorithms even when sparsification is used in the benchmarks
but \emph{not} in BCG!  This effect is not surprising, as BCG adds new
vertices to the active vertex set only when really necessary for ensuring
further progress in the optimization.

Two representative examples are shown in Table~\ref{tab:sparse}, where
we report the effect of sparsification in the size of the active set
as well as the increase in objective function value.

We also compared evolution of the function value and size of the
active set.  BCG decreases function value much more for the same
number of vertices because, by design, it performs more descent on a
given active set; see Figure~\ref{fig:sparsity}.

\begin{table*}
  \centering
  \begin{tabular}{lrrrr}
    \toprule
& vanilla & \ref{item:promoteDrop} &
                                                     \ref{item:promoteDrop},
                                                     \ref{item:postOpt}
  & \(\Delta f(x)\)\\
\midrule
PCG & \(112\) & \(62\) & \(60\) & \(2.6\%\) \\
LPCG & \(94\) & \(70\) & \(64\) & \(0.1\%\)\\
BCG & \(60\) & \(59\) & \(40\) & \(0.0\%\)\\
\bottomrule
\end{tabular} \qquad
\begin{tabular}{lrrr}
  \toprule
& vanilla &
                                                     \ref{item:promoteDrop},
                                                     \ref{item:postOpt}
  & \(\Delta f(x)\) \\
\midrule
ACG & \(300\) & \(298\) & \(7.4\%\) \\
PCG & \(358\) & \(255\) & \(8.2\%\) \\
BCG & \(211\) & \(211\) & \(0.0\%\) \\
\bottomrule
\end{tabular}
  \caption{Size of active set  and percentage increase in function value after sparsification. (No sparsification performed for BCG.) Left: Video Co-localization over
    \texttt{netgen\_08a}. Since we use LPCG and PCG as benchmarks, we
    report \ref{item:promoteDrop} separately as well. Right: Matrix
    Completion over \texttt{movielens100k} instance.  BCG without
    sparsification provides sparser solutions than the baseline
    methods with sparsification.
    In the last column, we report the percentage increase in objective
    function value due to sparsification. (Because this quantity is
    not affine invariant, this value should serve only to rank the
    quality of solutions.)}
  \label{tab:sparse}
\end{table*}

\subsection{Blending with pairwise steps}
\label{sec:blend-with-lazify}

Algorithm~\ref{alg:LOLCG} mixes descent steps with Frank–Wolfe steps.
One might be tempted to replace the Frank–Wolfe steps with (seemingly
stronger) pairwise steps, as the information needed for the latter
steps is computed in any case. In our tests, however, this variant did not
substantially differ in practical performance from the one that uses
the standard Frank–Wolfe step (see
Figure~\ref{fig:comp-bcg-fbcg pairwise}).  The explanation is that BCG
uses descent steps that typically provide better directions than
either Frank–Wolfe steps or pairwise steps. When the pairwise gap over
the active set is small, the Frank–Wolfe and pairwise directions
typically offer a similar amount of reduction in $f$.

\section{Computational experiments}
\label{sec:comp-results}

To compare our experiments to previous work, we used problems and
instances similar to those in
\cite{FW-converge2015,LDLCC2016,rao2015forward,
  braun2016lazifying,lan2017conditional}.  These problems include structured
regression, sparse regression, video co-localization, sparse signal
recovery, matrix completion, and Lasso. In particular, we compared our
algorithm to the Pairwise Frank–Wolfe algorithm from
\cite{FW-converge2015,LDLCC2016} and the lazified Pairwise Frank–Wolfe
algorithm from \cite{braun2016lazifying}.
Figure~\ref{fig:example} summarizes our results on four test problems.

We also benchmarked against the lazified
versions of the vanilla Frank–Wolfe and the Away-step Frank–Wolfe as
presented in \cite{braun2016lazifying} for completeness. We
implemented our code in \texttt{Python 3.6} using \texttt{Gurobi} (see
\cite{optimization2016gurobi})
as the LP solver for complex feasible regions;
as well as
obvious direct implementations for the probability
simplex, the cube and the \(\ell_1\)-ball.
As feasible regions, we used instances from
MIPLIB2010 (see \cite{koch2011miplib}), as done before in
\cite{braun2016lazifying}, along with some of the examples in
\cite{bashiri2017decomposition}. Code is available at
\url{https://github.com/pokutta/bcg}.
We used quadratic objective functions for the tests with random
coefficients, making sure that the global minimum lies outside the
feasible region, to make the optimization problem non-trivial; see
below in the respective sections for more details.

Every plot contains four diagrams depicting results of a single
instance.  The upper row measures progress in the logarithm of the
function value, while the lower row does so in the logarithm of the
gap estimate.  The left column measures performance in the number of
iterations, while the right column does so in wall-clock time. In the
graphs we will compare various algorithms denoted by the following
abbreviations: Pairwise Frank–Wolfe (PCG), Away-step Frank–Wolfe
(ACG), (vanilla) Frank–Wolfe (CG), blended conditional gradients
(BCG); we indicate the lazified versions of \cite{braun2016lazifying}
by prefixing with an \lq{}L\rq{}.  All tests were conducted with an
instance-dependent, fixed time limit, which can be easily read off the
plots.

The value $\Phi_t$ provided by the algorithm is an estimate of the
primal gap $f(x_t) - f(x^*)$.  The lazified versions (including
BCG) use it to estimate the required stepwise progress, halving it
occasionally, which provides a stair-like appearance in the graphs for
the dual progress. Note that if the certification in the
weak-separation oracle that $c(z-x) \ge \Phi$ for all $z \in P$ is
obtained from the original LP oracle (which computes the actual
optimum of $cy$ over $y \in P$), then we update the gap estimate
$\Phi_{t+1}$ with that value; otherwise the oracle would continue to
return \textbf{false} anyway until $\Phi$ drops below that value.  For
the non-lazified algorithms, we plot the dual gap \(\max_{v \in P}
\nabla f(x_{t}) (x_{t} - v)\).

\subsection*{Performance comparison}
\label{sec:perf-comp-bcg}
We implemented Algorithm~\ref{alg:LOLCG} as outlined above and used
SiGD (Algorithm~\ref{alg:simplex-descent}) for the descent steps as described in
Section~\ref{sec:proj-grad-desc-projfree}. For line search in
Line~\ref{line:simplex-line-search} of
Algorithm~\ref{alg:simplex-descent}, we perform standard backtracking,
and for Line~\ref{line:line-search} of Algorithm~\ref{alg:LOLCG}, we
do ternary search. In Figure~\ref{fig:example}, each of the four plots
itself contains four subplots depicting results of four variants of CG
on a single instance.  The two subplots in each upper row measure
progress in the logarithm (to base 2) of the function value, while the
two subplots in each lower row report the logarithm of
the gap estimate $\Phi_t$ from Algorithm~\ref{alg:LOLCG}. The two
subplots in the left column of each plot report performance in terms
of number of iterations, while the two subplots in the right column
report wall-clock time.

\begin{figure*}[ht]
  \centering
  \sidebyside%
  {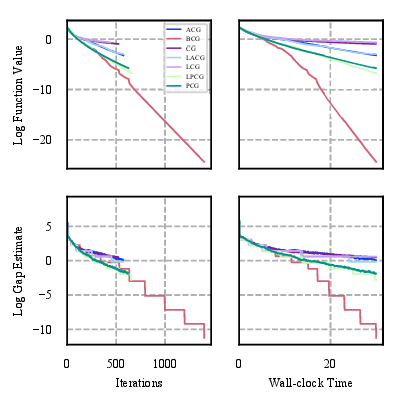}%
  {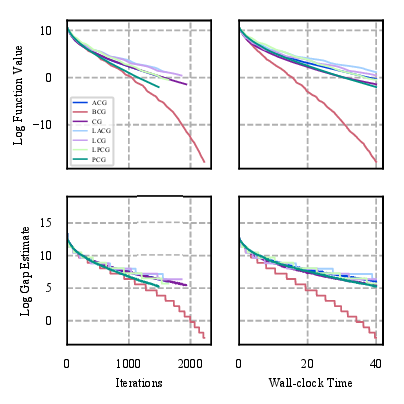}
  \sidebyside%
  {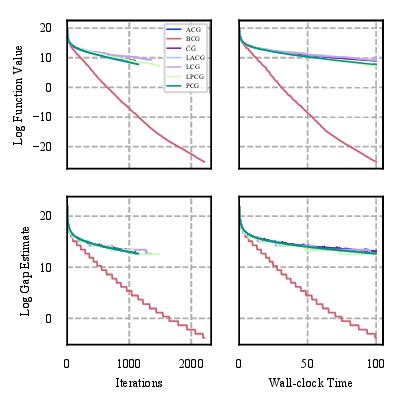}%
  {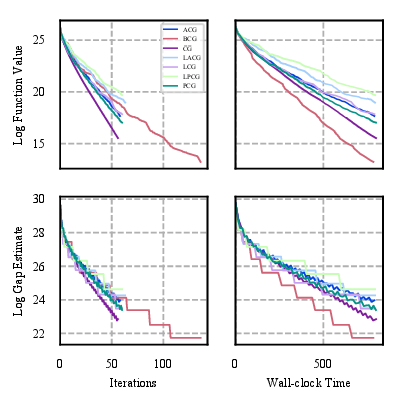}
  \caption{Four representative examples.
    (Upper-left)  Sparse signal recovery:
    \(\min_{x \in \R^n: \norm[1]{x} \leq \tau} \norm[2]{y - \Phi
      x}^2\),
    where \(\Phi\) is of size \(1000 \times 3000\) with density \(0.05\).    BCG made
\(1402\) iterations with \(155\) calls to the weak-separation oracle $\LPsep{P}$. The final solution
  is a convex combination of \(152\) vertices.
    (Upper-right)  Lasso. We solve \(\min_{x \in P} \norm{Ax
      - b}^2\) with \(P\) being the (scaled) \(\ell_1\)-ball.
    \(A\) is a \(400 \times 2000\) matrix with \(100\)
non-zeros. BCG made \(2130\) iterations, calling $\LPsep{P}$ \(477\)
times, with the final solution being a convex combination of \(462\)
vertices.
    (Lower-left) Structured regression over the Birkhoff polytope of dimension \(50\). BCG made
  \(2057\) iterations with \(524\) calls to $\LPsep{P}$. The  final solution
  is a convex combination of \(524\) vertices.
    (Lower-right)  Video co-localization over \texttt{netgen\_12b}
    polytope with an underlying \(5000\)-vertex graph.
  BCG made \(140\) iterations, with  \(36\) calls to $\LPsep{P}$.
  The final solution is a convex combination of \(35\) vertices.
  }
  \label{fig:example}
\end{figure*}



\paragraph{Lasso.} We tested BCG on lasso instances and compared them
to vanilla Frank–Wolfe, Away-step Frank–Wolfe, and Pairwise
Frank–Wolfe.  We generated Lasso instances similar to
\cite{FW-converge2015}, which has also also been used by several
follow-up papers as benchmark. Here we solve
\(\min_{x \in P} \norm{Ax - b}^2\) with \(P\) being the (scaled)
\(\ell_1\)-ball. We considered instances of varying sizes and the
results (as well as details about the instance) can be found in
Figure~\ref{fig:lasso}. Note that we did not benchmark any of the
lazified versions of \cite{braun2016lazifying} here, because the linear
programming oracle is so simple that lazification is not beneficial
and we used the LP oracle directly.

\begin{figure*}[htb]
  \centering
  \sidebyside%
  {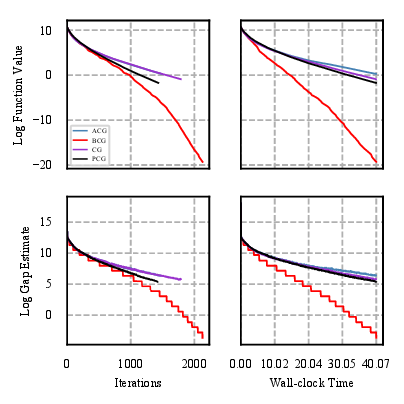}%
  {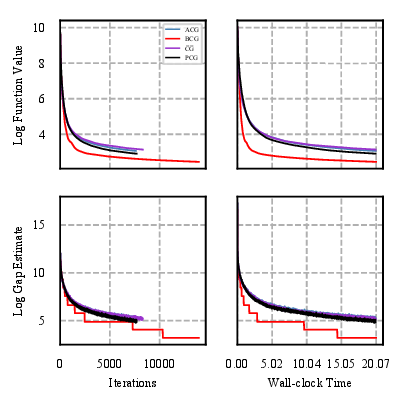}
  \sidebyside%
  {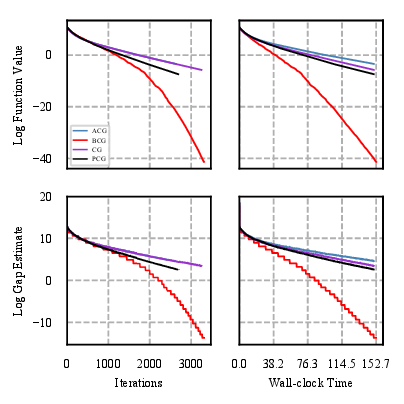}%
  {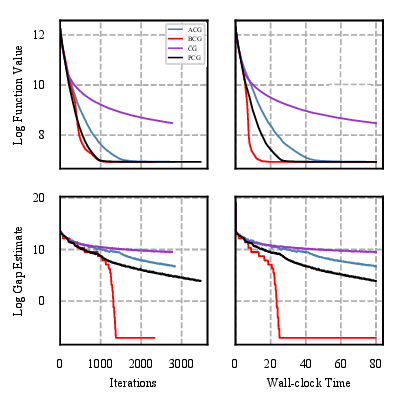}
\caption{Comparison of BCG, ACG, PCG and CG on Lasso
  instances.
Upper-left: \(A\) is a \(400 \times 2000\) matrix with \(100\)
non-zeros. BCG made \(2130\) iterations, calling the LP oracle \(477\)
times, with the final solution being a convex combination of \(462\)
vertices giving the sparsity.
Upper-right: \(A\) is a \(200 \times 200\) matrix with \(100\)
non-zeros. BCG made \(13952\) iterations, calling the LP oracle \(258\)
times, with the final solution being a convex combination of \(197\)
vertices giving the sparsity.
Lower-left: \(A\) is a \(500 \times 3000\) matrix with \(100\)
non-zeros. BCG made \(3314\) iterations, calling the LP oracle \(609\)
times, with the final solution being a convex combination of \(605\)
vertices giving the sparsity.
Lower-right: \(A\) is a \(1000 \times 1000\) matrix with \(200\)
non-zeros. BCG made \(2328\) iterations, calling the LP oracle \(1007\)
times, with the final solution being a convex combination of \(526\)
vertices giving the sparsity.
}
  \label{fig:lasso}
\end{figure*}

\paragraph{Video co-localization instances.} We also tested BCG on
video co-localization instances as done in \cite{FW-converge2015}. It
was shown in \cite{joulin2014efficient} that video co-localization
can be naturally reformulated as optimizing a quadratic function over
a flow (or path) polytope. To this end, we run tests on the same flow polytope
instances as used in \cite{lan2017conditional} (obtained from \url{http://lime.cs.elte.hu/~kpeter/data/mcf/road/}). We depict the results
in Figure~\ref{fig:videocolocalization}.

\begin{figure*}[htb]
  \centering
  \sidebyside%
  {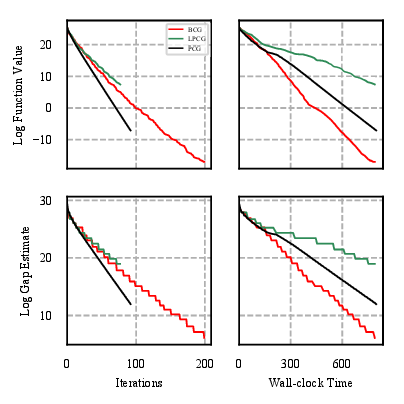}%
  {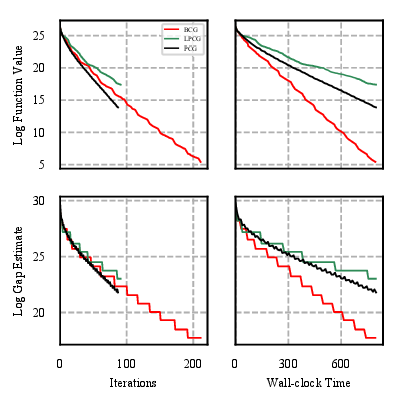}
  \sidebyside%
  {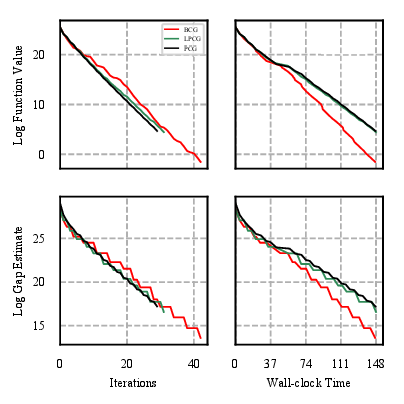}%
  {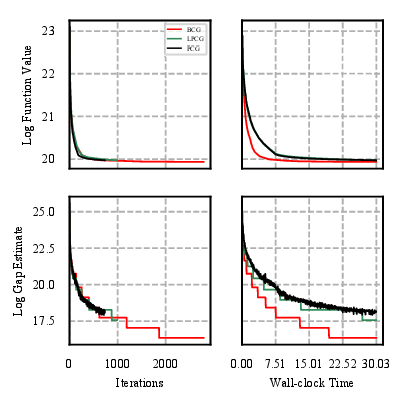}
\caption{Comparison of PCG, Lazy PCG,
  and BCG on video co-localization instances. Upper-Left:
  \texttt{netgen\_12b} for a \(3000\)-vertex graph.
BCG made \(202\) iterations, called $\LPsep{P}$
  \(56\) times and the final solution is a convex combination of
  \(56\) vertices. Upper-Right: \texttt{netgen\_12b} over
  a \(5000\)-vertex graph.
  BCG did \(212\) iterations, $\LPsep{P}$ was talked \(58\) times, and
  the final solution is a convex combination of \(57\) vertices.
Lower-Left: \texttt{road\_paths\_01\_DC\_a} over a
  \(2000\)-vertex graph.
  Even on instances where lazy PCG gains little advantage
  over PCG, BCG performs significantly better with empirically higher
  rate of convergence. BCG made \(43\) iterations, $\LPsep{P}$ was called
  \(25\) times, and the final convex combination has \(25\) vertices
Lower-Right:
  \texttt{netgen\_08a} over a \(800\)-vertex graph.
  BCG made \(2794\)
  iterations, $\LPsep{P}$ was called \(222\) times, and the final convex
  combination has \(106\) vertices. }
  \label{fig:videocolocalization}
\end{figure*}

\paragraph{Structured regression.}

We also compared BCG against PCG and LPCG on structured regression
problems, where we minimize a quadratic objective function over
polytopes corresponding to hard optimization problems used as
benchmarks in e.g.,
\cite{braun2016lazifying,lan2017conditional,bashiri2017decomposition}.
The polytopes were taken from MIPLIB2010 (see
\cite{koch2011miplib}). Additionally, we compare ACG, PCG, and
vanilla CG over the Birkhoff polytope for which linear
optimization is fast, so that there is little gain to be expected from
lazification.
See Figures~\ref{fig:structRegress} and~\ref{fig:structRegressBirkhoff}
for results.

\begin{figure*}[htb]
  \centering
\sidebyside%
  {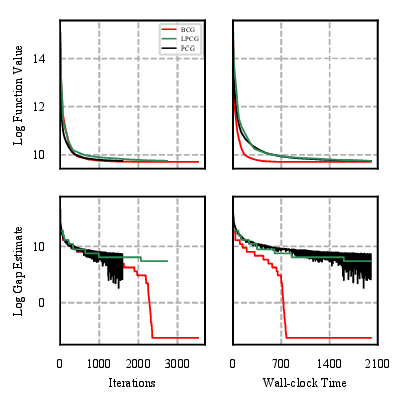}%
  {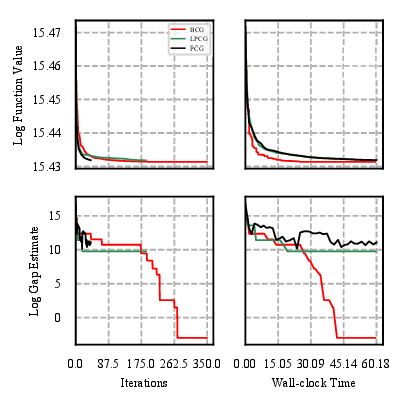}
\sidebyside%
  {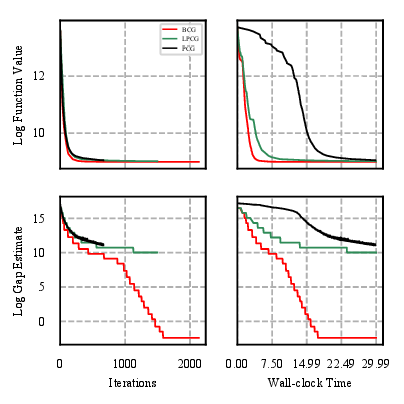}%
  {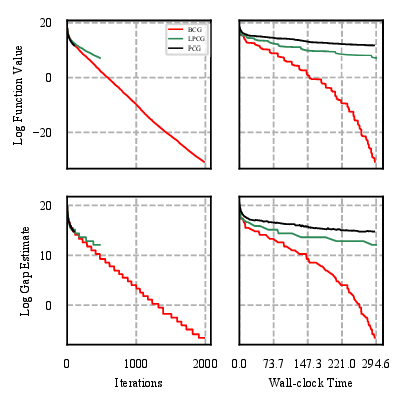}
\caption{Comparison of BCG, LPCG and PCG on structured regression
  instances.
Upper-Left: Over the \texttt{disctom} polytope. BCG made
  \(3526\) iterations with \(1410\) $\LPsep{P}$ calls and the final solution
  is a convex combination of \(85\) vertices.
Upper-Right: Over a
  \texttt{maxcut} polytope over a graph with \(28\) vertices. BCG made
  \(76\) $\LPsep{P}$ calls and the final solution is a convex combination of
  \(13\) vertices.
Lower-Left: Over the \texttt{m100n500k4r1}
  polytope.  BCG made
  \(2137\) iterations with \(944\) $\LPsep{P}$ calls and the final solution
  is a convex combination of \(442\) vertices.
Lower-right: Over the spanning tree
  polytope over the complete graph with \(10\) nodes.  BCG made
  \(1983\) iterations with \(262\) $\LPsep{P}$ calls and the final solution
  is a convex combination of \(247\) vertices.
BCG outperforms LPCG and PCG, even in the cases where LPCG
  is much faster than PCG.}
  \label{fig:structRegress}
\end{figure*}

\begin{figure*}[htb]
  \centering
\sidebyside%
  {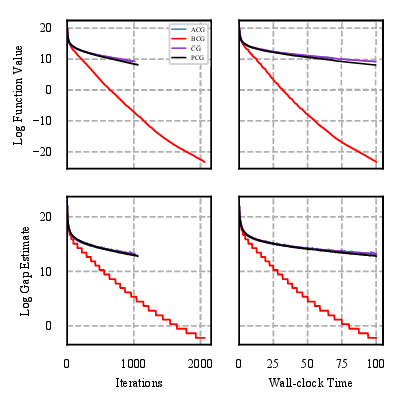}%
  {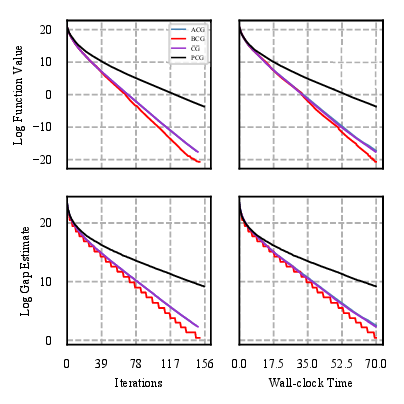}
\sidebyside%
  {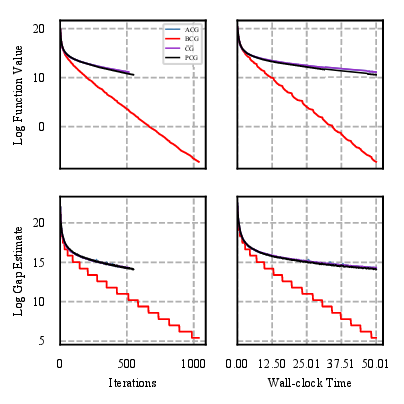}%
  {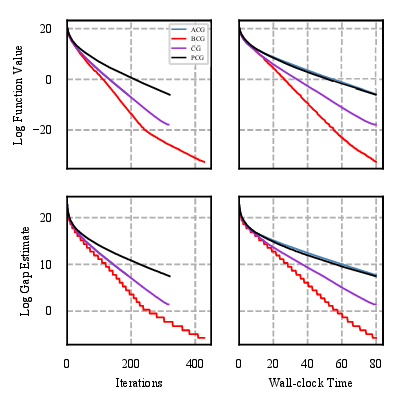}
\caption{Comparison of BCG, ACG, PCG and CG over the Birkhoff polytope.
Upper-Left: Dimension \(50\). BCG made
  \(2057\) iterations with \(524\) $\LPsep{P}$ calls and the final solution
  is a convex combination of \(524\) vertices.
Upper-Right: Dimension \(100\). BCG made
  \(151\) iterations with \(134\) $\LPsep{P}$ calls and the final solution
  is a convex combination of \(134\) vertices.
Lower-Left: Dimension \(50\).  BCG made
  \(1040\) iterations with \(377\) $\LPsep{P}$ calls and the final solution
  is a convex combination of \(377\) vertices.
Lower-right: Dimension \(80\).  BCG made
  \(429\) iterations with \(239\) $\LPsep{P}$ calls and the final solution
  is a convex combination of \(239\) vertices.
BCG outperforms  ACG, PCG and CG in all cases.}
  \label{fig:structRegressBirkhoff}
\end{figure*}

\paragraph{Matrix completion.}
\label{sec:matrix-completion}
Clearly, our algorithm also works directly over compact convex sets,
even though with a weaker theoretical bound of \(O(1/\varepsilon)\) as
convex sets need not have a pyramidal width bounded away from \(0\),
and linear optimization might dominate the cost, and hence the
advantage of lazification and BCG might be even greater empirically.

To this end, we also considered Matrix Completion instances over
the spectrahedron \(S = \{X \succeq 0 : \trace{X} = 1\} \subseteq
\R^{n \times n}\), where we solve the problem:
\[\min_{X \in S} \sum_{(i, j) \in L} (X_{i, j} - T_{i, j})^2,\]
where \(D = \{T_{i, j} \mid (i, j) \in L\} \subseteq \R\) is a data
set. In our tests we used the data sets Movie Lens 100k and Movie Lens
1m from \url{https://grouplens.org/datasets/movielens/} We subsampled
in the 1m case to generate \(3\) different instances.

As in the case of the Lasso benchmarks, we benchmark against ACG, PCG,
and CG, as the linear programming oracle is simple and there is no
gain to be expected from lazification. In the case of matrix
completion, the performance of BCG is quite comparable to ACG, PCG,
and CG in iterations, which makes sense over the spectrahedron,
because the gradient approximations computed by the linear
optimization oracle are essentially identical to the actual gradient,
so that there is no gain from the blending with descent steps. In
wall-clock time, vanilla CG performs best as the algorithm has the
lowest implementation overhead beyond the oracle calls compared to
BCG, ACG, and PCG (see Figure~\ref{fig:matrixCompletion}) and in
particular does not have to maintain the (large) active set.

\begin{figure*}[htb]
  \centering
\sidebyside%
  {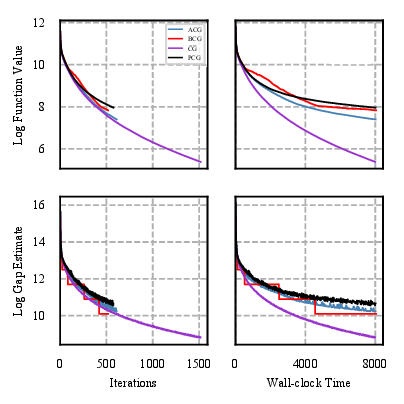}%
  {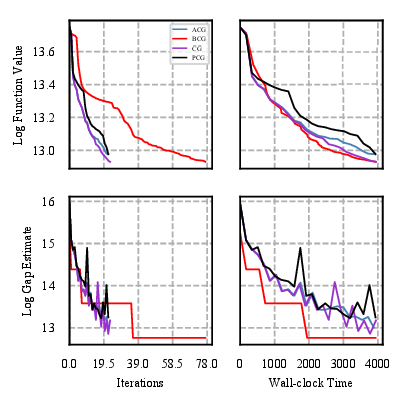}
\caption{Comparison of BCG, ACG, PCG and CG on matrix
  completion instances over
  the spectrahedron.
Upper-Left: Over the movie lens 100k data set. BCG made
  \(519\) iterations with \(346\) $\LPsep{P}$ calls and the final solution
  is a convex combination of \(333\) vertices.
Upper-Right: Over a subset of movie lens 1m data set. BCG made
  \(78\) iterations with \(17\) $\LPsep{P}$ calls and the final solution
  is a convex combination of \(14\) vertices.
BCG performs very similar to ACG, PCG, and vanilla CG as discussed.}
  \label{fig:matrixCompletion}
\end{figure*}

\paragraph{Sparse signal recovery.}
\label{sec:sparse-sign-recov}
We also performed computational experiments on the sparse signal
recovery instances from \cite{rao2015forward}, which have the
following form:
\[\hat x = \argmin_{x \in \R^n: \norm[1]{x} \leq \tau} \norm[2]{y -
    \Phi x}^2.\]
We chose a variety of parameters in our tests, including one test that
matches the setup in \cite{rao2015forward}. As in the case of the
Lasso benchmarks, we benchmark against ACG, PCG, and CG, as the linear
programming oracle is simple and there is no gain to be expected from
lazification. The results are shown in
Figure~\ref{fig:signalRecovery}.

\begin{figure*}[htb]
  \centering
\sidebyside%
  {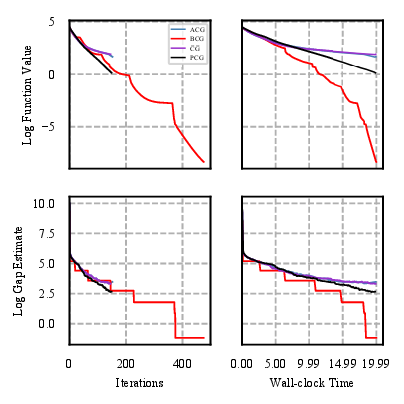}%
  {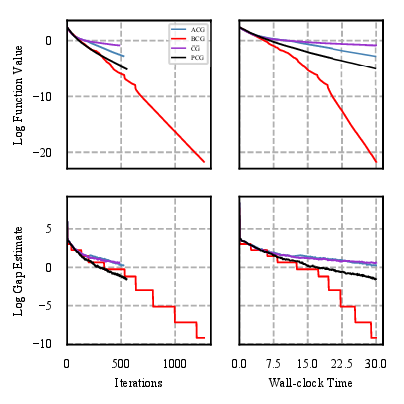}
\sidebyside%
  {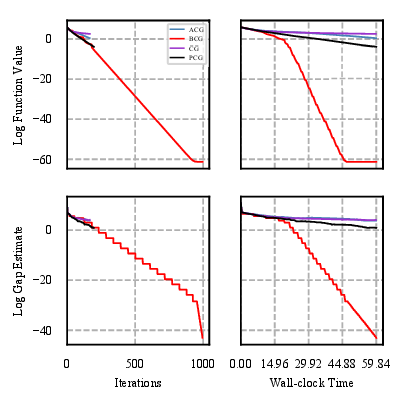}%
  {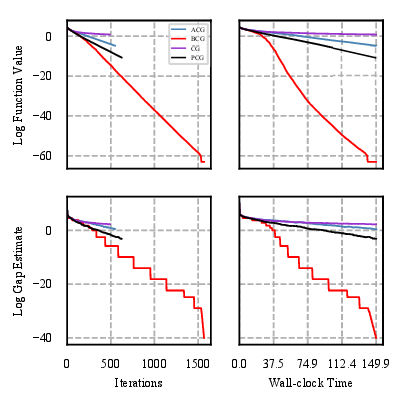}
\caption{Comparison of BCG, ACG, PCG and CG on a sparse signal
  recovery problem.
Upper-Left: Dimension is \(5000 \times 1000\) density is \(0.1\). BCG made
  \(547\) iterations with \(102\) $\LPsep{P}$ calls and the final solution
  is a convex combination of \(102\) vertices.
Upper-Right: Dimension is \(1000 \times 3000\) density is \(0.05\). BCG made
  \(1402\) iterations with \(155\) $\LPsep{P}$ calls and the final solution
  is a convex combination of \(152\) vertices.
Lower-Left: Dimension is \(10000 \times 1000\) density is \(0.05\). BCG made
  \(997\) iterations with \(87\) $\LPsep{P}$ calls and the final solution
  is a convex combination of \(52\) vertices.
Lower-right: dimension is \(5000 \times 2000\) density is \(0.05\). BCG made
  \(1569\) iterations with \(124\) $\LPsep{P}$ calls and the final solution
  is a convex combination of \(103\) vertices.
BCG outperforms all other algorithms in all examples significantly.}
  \label{fig:signalRecovery}
\end{figure*}

\subsection*{PGD vs.\ SiGD as subroutine}
\label{sec:pgd-as-sigd}

To demonstrate the superiority of SiGD over PGD we also tested two
implementations of BCG, once with standard PGD as subroutine and once
with SiGD as subroutine. The results can be found in
Figure~\ref{fig:comp-bcg-fbcg pairwise} (right): while PGD and SiGD compare
essentially identical in per-iteration progress, in terms of wall
clock time the SiGD variant is much faster. For comparison, we also
plotted LPCG on the same instance.

\begin{figure*}[htb]
  \centering
  \sidebyside%
  {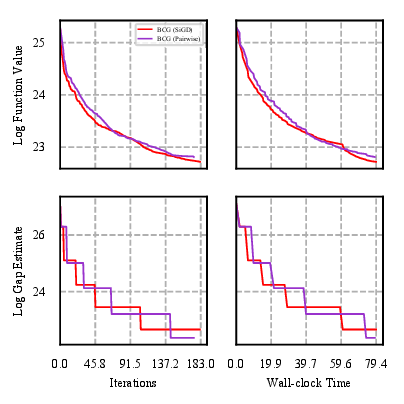}%
  {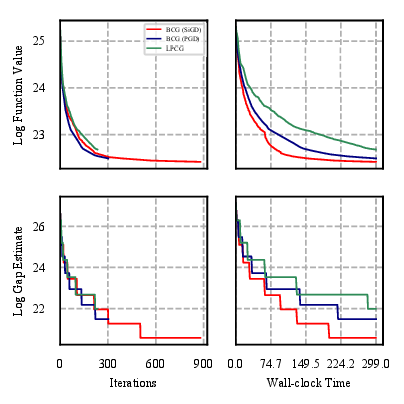}
  \caption{Comparison of BCG variants on a small
    video co-localization instance (instance \texttt{netgen\_10a}). Left: BCG with vanilla
    Frank–Wolfe steps (red) and with pairwise steps (purple). Performance is essentially
    equivalent here which matches our observations on other
    instances.
    Right: Comparison of oracle implementations PGD and SiGD. SiGD is
    significantly faster in wall-clock time.}
  \label{fig:comp-bcg-fbcg pairwise}
\end{figure*}

\subsection*{Pairwise steps vs.\ Frank–Wolfe steps}
\label{sec:pairwise-steps-vs}
As pointed out in Section~\ref{sec:blend-with-lazify}, a natural
extension is to replace the Frank–Wolfe steps in
Line~\ref{line:line-search} of Algorithm~\ref{alg:LOLCG} with pairwise
steps, since the information required is readily available. In
Figure~\ref{fig:comp-bcg-fbcg pairwise} (left) we depict
representative behavior: Little to no advantage when taking the more
complex pairwise step. This is expected as the Frank–Wolfe steps
are only
needed to add new vertices as the drop steps are subsumed the steps
from \Sora.
Note that BCG with Frank–Wolfe steps is
slightly faster per iteration, allowing for more steps within the time limit.

\subsection*{Comparison between lazified variants and BCG}
\label{sec:comp-betw-lazif}

For completeness, we also ran tests for BCG against various other
lazified variants of conditional gradient descent. The results are
consistent with our observations from before which we depict in
Figure~\ref{fig:comp-bcg-vs-lazified}.

\begin{figure*}[htb]
  \centering
  \sidebyside%
  {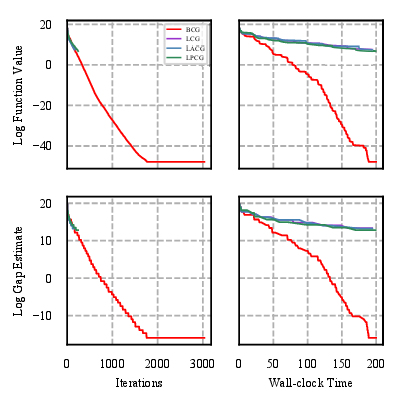}%
  {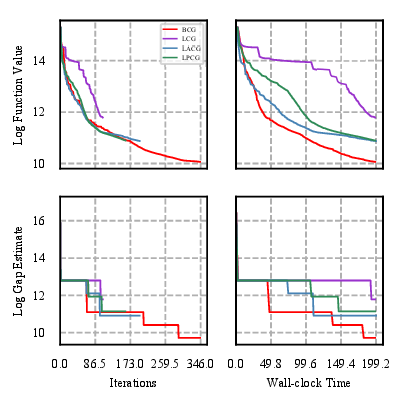}
\caption{Comparison of BCG, LCG,
  ACG, and PCG. Left: Structured regression instance
  over the spanning tree
  polytope over the complete graph with \(11\) nodes demonstrating significant
  performance difference in improving the function value and closing the dual gap; BCG made \(3031\)
  iterations, $\LPsep{P}$ was called \(1501\) times (almost always terminated
  early) and final solution is a convex combination of \(232\) vertices
  only. Right: Structured regression over the \texttt{disctom} polytope; BCG
  made \(346\) iterations, $\LPsep{P}$ was called \(71\) times, and final
  solution is a convex combination of \(39\) vertices only.
  Observe that not only the function value decreases faster,
  but the gap estimate, too.}
  \label{fig:comp-bcg-vs-lazified}
\end{figure*}

\subsection*{Standard vs.\ accelerated version}
\label{sec:stand-vs-accel}

Another natural variant of our algorithm is to replace \Sora
with its accelerated variant (both possible for PGD and
SiGD). As expected, due to the small size of the subproblem, we did
not observe any significant speedup from acceleration; see
Figure~\ref{fig:accelerationComparison}.

\begin{figure*}[htb]
  \centering
  \sidebyside%
  {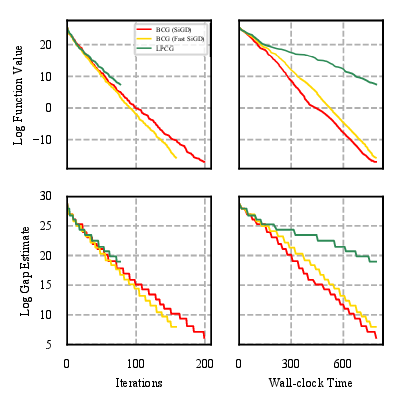}%
  {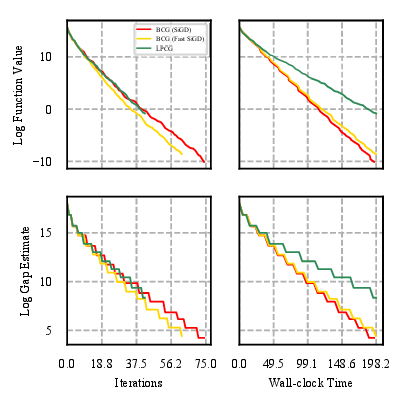}
\caption{Comparison of BCG, accelerated BCG and LPCG. Left: On a
  medium size video co-localization instance (\texttt{netgen\_12b}). Right: On a larger
  video co-localization instance (\texttt{road\_paths\_01\_DC\_a}). Here the accelerated version is (slightly) better in iterations but not in wall-clock
  time though. These findings are representative of all our other tests.}
  \label{fig:accelerationComparison}
\end{figure*}

\subsection*{Comparison to Fully-Corrective Frank–Wolfe}

As mentioned in the introduction, BCG is quite different from
FCFW. BCG is much faster and, in fact, FCFW is usually already
outpeformed by the much more efficient Pairwise-step CG (PCG), except
in some special cases. In Figure~\ref{fig:fcfwComp}, the left column
compares FCFW and BCG \emph{only across those iterations where FW
  steps were taken}; for completeness, we also implemented a variant
\emph{FCFW (fixed steps)} where only a fixed number of descent steps
in the correction subroutine are performed. As expected FCFW has a
better ``per-FW-iteration performance,'' because it performs
\emph{full} correction. The excessive cost of FCFW's correction
routine shows up in the wall-clock time (right column), where FCFW is
outperformed even by vanilla pairwise-step CG. This becomes even more
apparent when the iterations in the correction subroutine are broken
out and reported as well (see middle column). For purposes of
comparison, BCG and FCFW used both SiGD steps in the subroutine. (This
actually gives an advantage to FCFW, as SiGD was not known until the
current paper.) The per-iteration progress of FCFW is poor, due to
spending many iterations to optimize over active sets that are
irrelevant for the optimal solution. Our tests highlight the fact that
correction steps do not have constant cost in practice.

\begin{figure*}
  \centering
\includegraphics[width=1\linewidth]{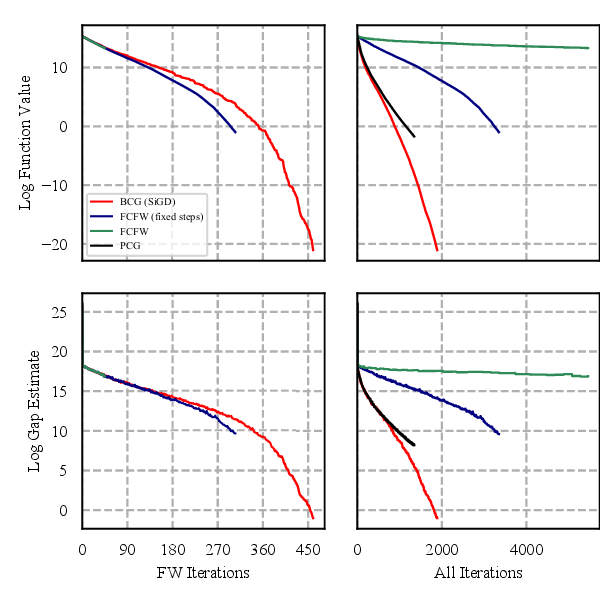}
\caption{\label{fig:fcfwComp} Comparison to FCFW across FW iterations,
  (all) iterations, and wall-clock time on a Lasso instance. Test run
  with 40s time limit. In this test we explicitly computed the dual
  gap of BCG, rather than using the estimate $\Phi_t$.}
\end{figure*}

\section{Final remarks}
\label{sec:final-remarks}

In \cite{lan2017conditional}, an accelerated method based on weak
separation and conditional gradient sliding was described.
This method provided
optimal tradeoffs between (stochastic) first-order oracle calls and
weak-separation oracle calls. An open question is whether the same
tradeoffs and acceleration could be realized by replacing SiGD
(Algorithm~\ref{alg:simplex-descent}) by an accelerated method.

After an earlier version of our work appeared online,
\cite{kerdreux2018restarting} introduced the \emph{Hölder Error Bound
  condition} (also known as \emph{sharpness} or the \emph{Łojasiewicz
  growth condition}).  This is a family of conditions parameterized by
$0 < p\leq 1$, interpolating between strongly convex (\(p=0\)) and
convex functions (\(p=1\)).  For such functions,
convergence rate $O(1/\varepsilon^p)$ has been shown for Away-step
Frank–Wolfe algorithms, among others.  Our analysis can be similarly
extended to objective functions satisfying this condition, leading to
similar convergence rates.

\section*{Acknowledgements}
\label{sec:acknowledgements}

We are indebted to Swati Gupta for the helpful discussions. Research reported in this paper was partially supported by NSF CAREER Award CMMI-1452463, and also
 NSF Awards 1628384, 1634597,
and 1740707; Subcontract 8F-30039
from Argonne National Laboratory; Award N660011824020 from the
DARPA Lagrange Program; and Award W911NF-18-1-0223 from the Army Research Office.


\bibliographystyle{abbrvnat}
\bibliography{bibliography}

\begin{thebibliography}{24}
\providecommand{\natexlab}[1]{#1}
\providecommand{\url}[1]{\texttt{#1}}
\expandafter\ifx\csname urlstyle\endcsname\relax
  \providecommand{\doi}[1]{doi: #1}\else
  \providecommand{\doi}{doi: \begingroup \urlstyle{rm}\Url}\fi

\bibitem[Bashiri and Zhang(2017)]{bashiri2017decomposition}
M.~A. Bashiri and X.~Zhang.
\newblock Decomposition-invariant conditional gradient for general polytopes
  with line search.
\newblock In \emph{Advances in Neural Information Processing Systems}, pages
  2687--2697, 2017.

\bibitem[Braun et~al.(2017)Braun, Pokutta, and Zink]{braun2016lazifying}
G.~Braun, S.~Pokutta, and D.~Zink.
\newblock Lazifying conditional gradient algorithms.
\newblock \emph{{Proceedings of ICML}}, 2017.

\bibitem[Frank and Wolfe(1956)]{frank1956algorithm}
M.~Frank and P.~Wolfe.
\newblock An algorithm for quadratic programming.
\newblock \emph{Naval research logistics quarterly}, 3\penalty0 (1-2):\penalty0
  95--110, 1956.

\bibitem[Freund and Grigas(2016)]{Freund2016}
R.~M. Freund and P.~Grigas.
\newblock New analysis and results for the {F}rank--{W}olfe method.
\newblock \emph{Mathematical Programming}, 155\penalty0 (1):\penalty0 199--230,
  2016.
\newblock ISSN 1436-4646.
\newblock \doi{10.1007/s10107-014-0841-6}.
\newblock URL \url{http://dx.doi.org/10.1007/s10107-014-0841-6}.

\bibitem[Freund et~al.(2017)Freund, Grigas, and Mazumder]{freund2017extended}
R.~M. Freund, P.~Grigas, and R.~Mazumder.
\newblock An extended {F}rank--{W}olfe method with ``in-face'' directions, and
  its application to low-rank matrix completion.
\newblock \emph{SIAM Journal on Optimization}, 27\penalty0 (1):\penalty0
  319--346, 2017.

\bibitem[Garber and Hazan(2013)]{garber2013linearly}
D.~Garber and E.~Hazan.
\newblock A linearly convergent conditional gradient algorithm with
  applications to online and stochastic optimization.
\newblock \emph{arXiv preprint arXiv:1301.4666}, 2013.

\bibitem[Garber and Meshi(2016)]{LDLCC2016}
D.~Garber and O.~Meshi.
\newblock Linear-memory and decomposition-invariant linearly convergent
  conditional gradient algorithm for structured polytopes.
\newblock \emph{arXiv preprint, arXiv:1605.06492v1}, May 2016.

\bibitem[Garber et~al.(2018)Garber, Sabach, and Kaplan]{garber2018fast}
D.~Garber, S.~Sabach, and A.~Kaplan.
\newblock Fast generalized conditional gradient method with applications to
  matrix recovery problems.
\newblock \emph{arXiv preprint arXiv:1802.05581}, 2018.

\bibitem[Gu{\'e}lat and Marcotte(1986)]{guelat1986some}
J.~Gu{\'e}lat and P.~Marcotte.
\newblock Some comments on wolfe's `away step'.
\newblock \emph{Mathematical Programming}, 35\penalty0 (1):\penalty0 110--119,
  1986.

\bibitem[{Gurobi Optimization}(2016)]{optimization2016gurobi}
{Gurobi Optimization}.
\newblock Gurobi optimizer reference manual version 6.5, 2016.
\newblock URL \url{https://www.gurobi.com/documentation/6.5/refman/}.

\bibitem[Gutman and Pe{\~n}a(2018)]{condition2018}
D.~H. Gutman and J.~F. Pe{\~n}a.
\newblock The condition of a function relative to a polytope.
\newblock \emph{arXiv preprint arXiv:1802.00271}, Feb. 2018.

\bibitem[Gutman and Pe{\~n}a(2019)]{condition2019}
D.~H. Gutman and J.~F. Pe{\~n}a.
\newblock The condition of a function relative to a set.
\newblock \emph{arXiv preprint arXiv:1901.08359}, Jan. 2019.

\bibitem[Holloway(1974)]{holloway1974extension}
C.~A. Holloway.
\newblock An extension of the {F}rank and {W}olfe method of feasible
  directions.
\newblock \emph{Mathematical Programming}, 6\penalty0 (1):\penalty0 14--27,
  1974.

\bibitem[Jaggi(2013)]{jaggi2013revisiting}
M.~Jaggi.
\newblock Revisiting {F}rank--{W}olfe: Projection-free sparse convex
  optimization.
\newblock In \emph{Proceedings of the 30th International Conference on Machine
  Learning (ICML-13)}, pages 427--435, 2013.

\bibitem[Joulin et~al.(2014)Joulin, Tang, and Fei-Fei]{joulin2014efficient}
A.~Joulin, K.~Tang, and L.~Fei-Fei.
\newblock Efficient image and video co-localization with frank-wolfe algorithm.
\newblock In \emph{European Conference on Computer Vision}, pages 253--268.
  Springer, 2014.

\bibitem[Kerdreux et~al.(2018{\natexlab{a}})Kerdreux, d'Aspremont, and
  Pokutta]{kerdreux2018restarting}
T.~Kerdreux, A.~d'Aspremont, and S.~Pokutta.
\newblock Restarting {F}rank--{W}olfe.
\newblock \emph{arXiv preprint arXiv:1810.02429}, 2018{\natexlab{a}}.

\bibitem[Kerdreux et~al.(2018{\natexlab{b}})Kerdreux, Pedregosa, and
  d'Aspremont]{kerdreux2018frank}
T.~Kerdreux, F.~Pedregosa, and A.~d'Aspremont.
\newblock {F}rank--{W}olfe with subsampling oracle.
\newblock \emph{arXiv preprint arXiv:1803.07348}, 2018{\natexlab{b}}.

\bibitem[Koch et~al.(2011)Koch, Achterberg, Andersen, Bastert, Berthold, Bixby,
  Danna, Gamrath, Gleixner, Heinz, Lodi, Mittelmann, Ralphs, Salvagnin, Steffy,
  and Wolter]{koch2011miplib}
T.~Koch, T.~Achterberg, E.~Andersen, O.~Bastert, T.~Berthold, R.~E. Bixby,
  E.~Danna, G.~Gamrath, A.~M. Gleixner, S.~Heinz, A.~Lodi, H.~Mittelmann,
  T.~Ralphs, D.~Salvagnin, D.~E. Steffy, and K.~Wolter.
\newblock {MIPLIB} 2010.
\newblock \emph{Mathematical Programming Computation}, 3\penalty0 (2):\penalty0
  103--163, 2011.
\newblock \doi{10.1007/s12532-011-0025-9}.
\newblock URL \url{http://mpc.zib.de/index.php/MPC/article/view/56/28}.

\bibitem[Lacoste-Julien and Jaggi(2015)]{FW-converge2015}
S.~Lacoste-Julien and M.~Jaggi.
\newblock On the global linear convergence of {F}rank--{W}olfe optimization
  variants.
\newblock In C.~Cortes, N.~D. Lawrence, D.~D. Lee, M.~Sugiyama, and R.~Garnett,
  editors, \emph{Advances in Neural Information Processing Systems}, volume~28,
  pages 496--504. Curran Associates, Inc., 2015.
\newblock URL
  \url{http://papers.nips.cc/paper/5925-on-the-global-linear-convergence-of-frank-wolfe-optimization-variants.pdf}.

\bibitem[Lan et~al.(2017)Lan, Pokutta, Zhou, and Zink]{lan2017conditional}
G.~Lan, S.~Pokutta, Y.~Zhou, and D.~Zink.
\newblock Conditional accelerated lazy stochastic gradient descent.
\newblock \emph{{Proceedings of ICML}}, 2017.

\bibitem[Lan(2017)]{Lan2017-MLnotes}
G.~G. Lan.
\newblock \emph{Lectures on Optimization for Machine Learning}.
\newblock ISyE, April 2017.

\bibitem[Levitin and Polyak(1966)]{levitin1966constrained}
E.~S. Levitin and B.~T. Polyak.
\newblock Constrained minimization methods.
\newblock \emph{USSR Computational mathematics and mathematical physics},
  6\penalty0 (5):\penalty0 1--50, 1966.

\bibitem[Nemirovski and Yudin(1983)]{nemirovskii1983problem}
A.~Nemirovski and D.~Yudin.
\newblock \emph{Problem complexity and method efficiency in optimization}.
\newblock Wiley, 1983.
\newblock ISBN 0-471-10345-4.

\bibitem[Rao et~al.(2015)Rao, Shah, and Wright]{rao2015forward}
N.~Rao, P.~Shah, and S.~Wright.
\newblock Forward--backward greedy algorithms for atomic norm regularization.
\newblock \emph{IEEE Transactions on Signal Processing}, 63\penalty0
  (21):\penalty0 5798--5811, 2015.

\end{thebibliography}

\appendix

\section{Upper bound on simplicial curvature}
\label{sec:upper-bound-simpl-curv}

\begin{lemma}
  \label{lem:upper-bound-simplex-smoothness}
  Let \(f \colon P \to \mathbb{R}\) be an \(L\)-smooth function
  over a polytope \(P\) with diameter \(D\) in some norm \(\norm{}\).
  Let \(S\) be a set of vertices of \(P\).
  Then the function \(f_{S}\) from Section~\ref{sec:proj-grad-desc}
  is smooth with smoothness parameter at most
  \begin{equation}
    \label{eq:upper-bound-simplex-smoothness}
    L_{f_{S}} \leq \frac{L D^{2} \size{S}}{4}
    .
  \end{equation}
\begin{proof}
Let \(S = \{v_{1}, \dots, v_{k}\}\).
Recall that
\(f_{S} \colon \Delta^{k} \to \mathbb{R}\) is defined
on the probability simplex via \(f_{S}(\alpha) \coloneqq f(A \alpha)\),
where \(A\) is the linear operator
defined via \(A \alpha \coloneqq \sum_{i=1}^{k} \alpha_{i} v_{i}\).
We need to show
\begin{equation}
  \label{eq:3}
  f_{S}(\alpha) - f_{S}(\beta) - \nabla f_{S}(\beta) (\alpha - \beta)
  \leq
  \frac{L D^{2} \size{S}}{8}
  \cdot
  \norm[2]{\alpha - \beta}^{2}
  ,
  \qquad
  \alpha, \beta \in \Delta^{k}
  .
\end{equation}
We start by expressing the left-hand side in terms of \(f\)
and applying
the smoothness of \(f\):
\begin{equation}
  \label{eq:4}
  f_{S}(\alpha) - f_{S}(\beta) - \nabla f_{S}(\beta) (\alpha - \beta)
  =
  f(A \alpha) - f(A \beta)
  - \nabla f(A \beta) \cdot (A \alpha - A \beta)
  \leq
  \frac{L}{2} \cdot \norm{A \alpha - A \beta}^{2}
  .
\end{equation}
Let \(\gamma_{+} \coloneqq \max \{\alpha - \beta, 0\}\)
and \(\gamma_{-} \coloneqq \max \{\beta - \alpha, 0\}\)
with the maximum taken coordinatewise.
Then \(\alpha - \beta = \gamma_{+} - \gamma_{-}\)
with \(\gamma_{+}\) and \(\gamma_{-}\) nonnegative vectors with
disjoint support.
In particular,
\begin{equation}
  \label{eq:5}
  \norm[2]{\alpha - \beta}^{2}
  =
  \norm[2]{\gamma_{+} - \gamma_{-}}^{2}
  =
  \norm[2]{\gamma_{+}}^{2} + \norm[2]{\gamma_{-}}^{2}
  .
\end{equation}

Let \(\allOne\) denote the vector of length \(k\) with all its
coordinates \(1\).  Since \(\allOne \alpha = \allOne \beta = 1\), we
have \(\allOne \gamma_{+} = \allOne \gamma_{-}\).
Let $t$ denote this last quantity, which is clearly nonnegative.
If \(t = 0\) then \(\gamma_{+} = \gamma_{-} = 0\) and
\(\alpha = \beta\), hence the claimed \eqref{eq:3} is obvious.  If \(t
> 0\) then \(\gamma_{+} / t\) and \(\gamma_{-} / t\) are points of the
simplex \(\Delta^{k}\), therefore
\begin{equation}
  \label{eq:6}
  D \geq \norm{A (\gamma_{+} / t) - A (\gamma_{-} / t)}
  = \frac{\norm{A \alpha - A \beta}}{t}
  .
\end{equation}
Using \eqref{eq:5} with \(k_{+}\) and \(k_{-}\) denoting the number of
non-zero coordinates of \(\gamma_{+}\) and \(\gamma_{-}\),
respectively, we obtain
\begin{equation}
  \label{eq:7}
  \norm[2]{\alpha - \beta}^{2}
  =
  \norm[2]{\gamma_{+}}^{2} + \norm[2]{\gamma_{-}}^{2}
  \geq
  t^{2}
  \left(
    \frac{1}{k_{+}} + \frac{1}{k_{-}}
  \right)
  \geq
  t^{2} \cdot \frac{4}{k_{+} + k_{-}}
  \geq
  \frac{4 t^{2}}{k}
  .
\end{equation}
By \eqref{eq:6} and \eqref{eq:7} we conclude that
\(\norm{A \alpha - A \beta}^{2} \leq
k D^{2} \norm[2]{\alpha - \beta}^{2} / 4\),
which together with \eqref{eq:4} proves the claim \eqref{eq:3}.
\end{proof}
\end{lemma}

\begin{lemma}
  \label{lem:curvature-by-simplicial}
  Let \(f \colon P \to \mathbb{R}\) be a convex function
  over a polytope \(P\) with finite simplicial curvature
  \(C^{\Delta}\).
  Then \(f\) has curvature at most
  \begin{equation}
    \label{eq:lem:curvature-by-simplicial}
    C \leq 2 C^{\Delta}.
  \end{equation}
\begin{proof}
Let \(x, y \in P\) be two distinct points of \(P\).
The line through \(x\) and \(y\) intersects \(P\)
in a segment \([w, z]\), where \(w\) and \(z\) are points on the
\emph{boundary} of \(P\), i.e., contained in facets of \(P\),
which have dimension \(\dim P - 1\).
Therefore by Caratheodory's theorem there are vertex sets \(S_{w}\),
\(S_{z}\) of \(P\) of size at most \(\dim P\) with
\(w \in \conv S_{w}\) and \(z \in \conv S_{z}\).
As such \(x,y \in \conv S\) with \(S \coloneqq S_{w} \cup S_{z}\) and
\(\card{S} \leq 2 \dim P\).

Reusing the notation from the proof of
Lemma~\ref{lem:upper-bound-simplex-smoothness},
let \(k \coloneqq \size{S}\) and \(A\) be a linear transformation
with \(S = \{A e_{1}, \dotsc, A e_{k}\}\)
and \(f_{S}(\zeta) = f(A \zeta)\) for all \(\zeta \in \Delta^{k}\).
Since \(x, y \in \conv S\), there are \(\alpha, \beta \in \Delta^{k}\)
with \(x = A \alpha\) and \(y = A \beta\).
Therefore by smoothness of \(f_{S}\) together with
\(L_{f_{S}} \leq C^{\Delta}\)
and \(\norm{\beta - \alpha} \leq \sqrt{2}\):
\begin{equation}
  \label{eq:1}
 \begin{split}
  f(\gamma y + (1 - \gamma) x) - f(x) - \gamma \nabla f(x) (y - x)
  &
  =
  f(\gamma A \beta + (1 - \gamma) A \alpha) - f(A \alpha)
  - \gamma \nabla f(A \alpha) \cdot (A \beta - A \alpha)
  \\
  &
  =
  f_{S}(\gamma \beta + (1 - \gamma) \alpha) - f_{S}(\alpha)
  - \gamma \nabla f_{S}(\alpha) (\beta - \alpha)
  \\
  &
  \leq
  \frac{L_{f_{S}} \norm{\gamma (\beta - \alpha)}^{2}}{2}
  =
  \frac{L_{f_{S}} \norm{\beta - \alpha}^{2}}{2} \cdot \gamma^{2}
  \leq C^{\Delta} \gamma^{2}
 \end{split}
\end{equation}
showing that \(C \leq 2 C^{\Delta}\) as claimed.
\end{proof}
\end{lemma}


\begin{figure*}[htb]
  \centering
  \sidebyside%
  {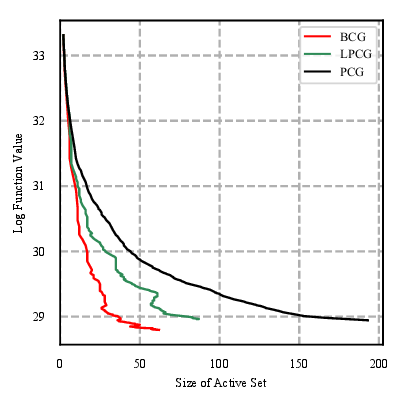}%
  {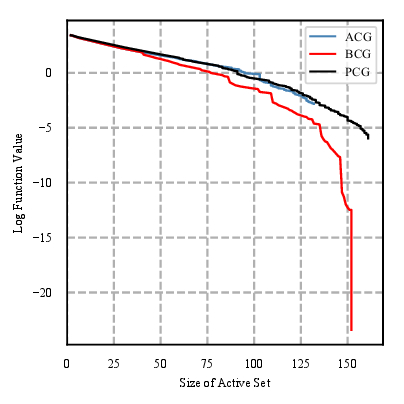}
\caption{Comparison of ACG, PCG and LPCG against BCG in
  function value and size of the active set.
Left: Video Co-Localization instance.
Right: Sparse signal recovery.
}
  \label{fig:sparsity}
\end{figure*}

\end{document}